\documentclass[11pt,twoside]{article}
\usepackage{amsmath}
\usepackage{amsthm}
\usepackage{amsfonts}
\usepackage{amssymb}
\usepackage{color}
\usepackage{mathrsfs}
\usepackage{cleveref}
\usepackage{cite}
\usepackage{geometry}
\usepackage{marginnote}
\usepackage{todonotes}
\allowdisplaybreaks
\newcommand{\RR}{\mathbb R}

\newcommand{\R}{\mathbb{R}}
\newcommand{\N}{\mathbb{N}}
\newcommand{\pO}{\partial\Omega}
\newcommand{\pa}{\partial}
\newcommand{\eps}{\varepsilon}
\newcommand{\var}{\varepsilon}

\newcommand{\supp}{{\rm supp} \, }

\newcommand{\wto}{\rightharpoonup}

\newcommand{\io}{\int_\Omega}
\newcommand{\bom}{\overline{\Omega}}

\newcommand{\nn}{\nonumber}
\newcommand{\dd}{\, \mathrm{d}}
\definecolor{cadmiumgreen}{rgb}{0.0, 0.42, 0.24}

\newcommand{\kg}{k}

\newcommand{\tge}{\gamma_\eta}
\newcommand{\IA}{\mathcal{A}^{-1}}
\newtheorem{theorem}{Theorem}[section]

\newtheorem{lemma}[theorem]{Lemma}
\newtheorem{proposition}[theorem]{Proposition}
\newtheorem{definition}[theorem]{Definition}
\newtheorem{remark}[theorem]{Remark}
\newcommand{\be}{\begin{equation} \label}
\newcommand{\ee}{\end{equation}}
\newcommand{\bea}{\begin{eqnarray}\label}
\newcommand{\eea}{\end{eqnarray}}
\newcommand{\bas}{\begin{eqnarray*}}
\newcommand{\eas}{\end{eqnarray*}}
\numberwithin{equation}{section}
\begin{document}
\title{Weak solutions to triangular cross diffusion systems modeling chemotaxis with local sensing} 

\author{
Laurent Desvillettes\footnote{desvillettes@math.univ-paris-diderot.fr}\\
{\small  Universit\'e de Paris, IUF and Sorbonne Universit\'e, CNRS, IMJ-PRG,}\\
{\small F-75013 Paris, France}
\and
Philippe Lauren\c{c}ot \footnote{laurenco@math.univ-toulouse.fr}\\
{\small  Institut de Math\'ematiques de Toulouse, UMR~5219, Universit\'e de Toulouse, CNRS,}\\
{\small F--31062 Toulouse Cedex 9, France}
\and
Ariane Trescases\footnote{ariane.trescases@math.univ-toulouse.fr}\\
{\small  Institut de Math\'ematiques de Toulouse, UMR~5219, Universit\'e de Toulouse, CNRS,}\\
{\small F--31062 Toulouse Cedex 9, France}
\and
Michael Winkler\footnote{michael.winkler@math.uni-paderborn.de}\\
{\small Institut f\"ur Mathematik, Universit\"at Paderborn,}\\
{\small 33098 Paderborn, Germany} 
}


\date{\today}
\maketitle

\begin{abstract}
New estimates and global existence results are provided for a class of systems of cross diffusion equations arising from the modeling of chemotaxis with local sensing, possibly featuring a growth term of logistic-type as well. For sublinear non-increasing motility functions, convergence to the spatially homogeneous steady state is shown, a dedicated Liapunov functional being constructed for that purpose.
\end{abstract}

%
%
\pagestyle{myheadings}
\markboth{\sc{L.~Desvillettes, Ph.~Lauren\c cot, A.~Trescases, M.~Winkler}}{\sc{Triangular cross diffusion systems modeling chemotaxis with local sensing}}
%
\section{Introduction}

We consider a class of systems of two parabolic equations in which the first equation is a cross diffusion equation (that is, the diffusion rate in this equation depends on the solution of the second equation), while the second equation is a standard heat equation coupled to the first one only through its source term. Such systems are sometimes called triangular cross diffusion systems.
We focus on the systems introduced in \cite{DKTY2019} to treat specific situations arising in the theory of chemotaxis. The quantity $u:=u(t,x) \ge 0$ is then the density of cell and $v:=v(t,x) \ge 0$ is the concentration of chemoattractant. We refer to 
\cite{DKTY2019} for a discussion of the modeling assumptions underlying such systems. Let us just say that with respect to general systems appearing in
the modeling of chemotaxis, where the dynamics of the density of cells is driven by the evolution equation $\partial_t u = \mathrm{div} ( \nabla F +  G)$, where $F, G$ both depend on $u$ and $v$, the specificity of the system considered here is that it can be written in a form where $G=0$; that is,  

\begin{subequations}\label{ivp}
	\begin{align}
		\partial_t u - \Delta (u \gamma(v)) = 0\,, \qquad &(t,x)\in (0,\infty)\times\Omega\,, \label{ivp1} \\
		\varepsilon \partial_t v - \Delta v= u - v\,, \qquad &(t,x)\in (0,\infty)\times\Omega\,, \label{ivp2} \\	
		\nabla (u \gamma(v)) \cdot \mathbf{n}  = \nabla v \cdot \mathbf{n} = 0\,, \qquad &(t,x)\in (0,\infty)\times\partial\Omega\,, \label{ivp3} \\
		(u,v)(0, \cdot) = (u^{in},v^{in})\,, \qquad &x\in \Omega\,, \label{ivp4}
\end{align}
\end{subequations}
where $\Omega$ is a smooth bounded domain of $\RR^N$, with $N\ge2$, and $\varepsilon>0$. Here, $\bf n$ is the outward unit normal vector at a point of $\partial\Omega$. The initial data $u^{in}$ and $v^{in}$ are given and nonnegative. Typical functions $\gamma$ are assumed to be bounded and strictly positive on $[0, \infty)$, and to decay at infinity (that is, when $v \to \infty$) typically like a power. In other words, they generalize the prototype (which makes sense from the point of view of modeling) given by
\begin{equation}\label{proto}
	\gamma(z)=(z+1)^{-\kg}, 
	\qquad z>0,
\end{equation}
for some $\kg>0$, but are not always assumed to be monotone decreasing.

\medskip

We also consider the counterpart of this system when the cell population has a logistic-type growth, that is,
\begin{subequations}\label{02}
\begin{align}
	\partial_t u- \Delta (u\gamma(v)) = u \, h(u),
	\qquad &(t,x)\in (0,\infty)\times\Omega, \\
	\varepsilon \partial_t v=\Delta v-v+u,
	\qquad &(t,x)\in (0,\infty)\times\Omega, \\
	\nabla (u \gamma(v)) \cdot \mathbf{n}  = \nabla v \cdot \mathbf{n} =0,
	\qquad & (t,x)\in (0,\infty)\times\partial\Omega, \\
	(u,v)(0, \cdot) = (u^{in},v^{in})\,,
	\quad & x\in\Omega,
\end{align}
\end{subequations}
where $h$ is a continuous function. It can indeed be interesting to take into account cells' division, as well as their death due to the lack of resources. 

\subsection{Notation}

We will sometimes denote the spaces $L^p(\Omega)$, $H^1(\Omega)$, and $(H^1(\Omega))'$ by  $L^p$, $H^1$, and $(H^1)'$ , respectively (with $p\in[1,\infty]$). Furthermore, for $w\in L^p$, we denote the $L^p$ norm of $w$ by $\|w\|_p$.

\medskip

Given $w\in (H^1)'(\Omega)$, we define $\langle w \rangle$ by
\begin{equation*}
	\langle w \rangle := \frac{1}{|\Omega|} \langle w, 1 \rangle_{(H^1)',H^1}
\end{equation*}
and note that
\begin{equation*}
	\langle w \rangle = \frac{1}{|\Omega|} \int_\Omega w(x)\ \mathrm{d}x\, {\hbox{ when }} \qquad w\in (H^1)'(\Omega) \cap L^1(\Omega)\,.
\end{equation*}
 For $w\in (H^1)'(\Omega)$ such that $\langle w \rangle =0$, we introduce $\mathcal{K}w\in H^1(\Omega)$ as the unique (variational) solution to
 \begin{equation}
 		-\Delta (\mathcal{K}w) = w \;\;\text{ in }\;\; \Omega\,, \qquad \nabla (\mathcal{K}w) \cdot \mathbf{n} = 0 \;\;\text{ on }\;\; \partial\Omega\,, \qquad \langle \mathcal{K}w \rangle = 0\,. \label{opk}
 \end{equation}
The operator $\mathcal{K}$ plays a significant role in the analysis of our system, in particular in view of the specific form of the cross-diffusion in \eqref{ivp1}. Indeed, for a solution $(u,v)$ regular enough, one expects that the conservation of mass holds for $u$, that is, $\langle u \rangle = \langle u^{in} \rangle$, and that consequently \eqref{ivp1} can be rewritten as
\begin{align}
		\partial_t \mathcal{K} \left( u - \langle u^{in} \rangle \right) = \langle u \gamma(v) \rangle - u \gamma(v) \,, \qquad &(t,x)\in (0,\infty)\times\Omega\,. \label{ivp1K}
\end{align}
For this reason, we choose the following norm on $(H^1)'(\Omega)$:
  \begin{equation}
 		w\in (H^1)'(\Omega) \mapsto  \|w\|_{(H^1)'} := \| \nabla \mathcal{K}(w-\langle w \rangle)\|_2\,. \label{H1pnorm}
 \end{equation}
 Still for $w\in (H^1)'(\Omega)$, not necessarily with zero average, we also define $\IA w \in H^1(\Omega)$ as the unique (variational) solution to
  \begin{equation}
 		-\Delta \IA w + \IA w  = w \;\;\text{ in }\;\; \Omega\,, \qquad \nabla \IA w\cdot \mathbf{n} = 0 \;\;\text{ on }\;\; \partial\Omega \,. \label{opa}
 \end{equation}
Clearly, $\IA$ is the extension to $(H^1)'(\Omega)$ of the inverse of the unbounded linear operator $\mathcal{A}$ on $L^2(\Omega)$ with domain
\begin{equation} 
	\begin{split}
		D(\mathcal{A}) & := \{ z \in H^2(\Omega)\ :\ \nabla z \cdot \mathbf{n} = 0  \;\;\text{ on }\;\; \partial\Omega\}\,, \\
		\mathcal{A}z & := - \Delta z + z \;\;\text{ for }\;\; z\in D(\mathcal{A})\,.
	\end{split}\label{OpA}
\end{equation}
We note that the following norm: 
 \begin{equation} w\in (H^1)'(\Omega) \mapsto \| \nabla \IA w \|_2 + \| \IA w -\langle  \IA w \rangle \|_2 = \| \nabla \IA w \|_2 + \| \IA w -\langle  w \rangle \|_2\,, \label{H1pnormA}
 \end{equation}
is equivalent to the $(H^1)'(\Omega)$-norm defined in \eqref{H1pnorm}.
 
\subsection{Main results}

We first propose a definition of very weak solutions associated with  problem~\eqref{ivp}:
\begin{definition}\label{dw2}
Let  $\Omega$ be a smooth bounded domain of $\RR^N$, with $N\ge2$, $\varepsilon>0$,
 and $\gamma \in C ([0,\infty);(0,\infty))$. Suppose that $u^{in}\in L^1(\Omega)$ and $v^{in}\in L^1(\Omega)$ are nonnegative, and that, for all $T>0$,
  \be{w12}
	u\in L^1((0,T)\times\Omega) \qquad \text{and} \qquad v\in L^1((0,T)\times\Omega),
\ee
  are such that $u\ge 0$ and $v>0$ a.e.~in $(0,\infty)\times\Omega$, and that, for all $T>0$,
  \be{w22}
	u\gamma(v) \in L^1((0,T)\times\Omega).
  \ee
  Then $(u,v)$ is called a {\em global very weak solution} of (\ref{ivp}) if
  \bea{wu2}
	- \int_0^\infty \io u\partial_t \varphi \dd x \dd t - \io u^{in} \varphi(0) \dd x
	= \int_0^\infty \io u\gamma(v)\Delta\varphi \dd x \dd t, \nn
  \eea
  and
  \bea{wv2}
	- \int_0^\infty \io \varepsilon v\partial_t\varphi \dd x \dd t - \io \varepsilon v^{in}\varphi(0) \dd x 
	= \int_0^\infty \io  v\, (\Delta\varphi - \varphi) \dd x \dd t 
	+ \int_0^\infty \io u\varphi \dd x \dd t, \nn
  \eea
hold for any $\varphi\in C_0^\infty([0,\infty)\times\bom)$ such that $\nabla\varphi\cdot\mathbf{n}=0$ on $(0,\infty)\times\pO$.
\end{definition}

We first show that an algebraic growth on $1/\gamma$ at infinity is sufficient to obtain the existence of a global very weak solution to the system~\eqref{ivp} without imposing any smoothness assumption on $\gamma$.

\begin{theorem}\label{theo26}
  Let $N\ge 2$ and $\Omega\subset\R^N$ be a bounded domain with smooth boundary,
 and assume that $\varepsilon>0$, that $\gamma$ is continuous and bounded on $[0,\infty)$, and that for some $K_1>0$ and $\kg\ge 0$,
\begin{equation}\label{phi1}
	\frac{1}{\gamma(z)} \le K_1 \, (z+1)^{\kg},
	\qquad z > 0.
 \end{equation}
\par

  Then for any choice of $(u^{in},v^{in})$ such that
  \be{init}
	\left\{ \begin{array}{l}
	u^{in} \in L^{p_0}(\Omega), \ p_0>N/2\,,
	\mbox{ is nonnegative and} \\[1mm]
	v^{in}\in L^\infty(\Omega) \mbox{ is nonnegative,}
	\end{array} \right.
  \ee
  there exists a global very weak solution $(u,v)$ of (\ref{ivp}) in the sense of Definition~\ref{dw2}. Furthermore, for all $q\in (1,\infty)$, $p\in (1,2)$, and $T>0$,
  \begin{align*} u &\in L^\infty((0,\infty);L^1(\Omega))\cap L^p((0,T)\times\Omega) \cap L^\infty((0,T);(H^1(\Omega))')  \, , \\
  v &\in L^\infty((0,\infty);L^2(\Omega)) \cap L^q((0,T)\times\Omega) \cap L^{2}((0,T);H^1(\Omega))\,,  \\
  u \sqrt{\gamma(v)} &\in L^2((0,T)\times\Omega)  \, .
  \end{align*}

Furthermore, if $v^{in}\in W^{1,q}(\Omega)$ for some $q\in (1,2)$, then $v\in L^\infty((0,T);W^{1,q}(\Omega))$ for all $T>0$.
\end{theorem}

Observe that, since $N\ge 2$, one has $N/2\ge 2N/(N+2)$, so that $L^{p_0}(\Omega)$ is continuously embedded in $(H^1)'(\Omega)$ and it follows from \eqref{init} that 
\begin{equation}
	u^{in}\in (H^1)'(\Omega). \label{initb}
\end{equation}

When $\gamma\in C^3([0,\infty))\cap L^\infty(0,\infty)$, existence of weak solutions to~\eqref{ivp} is shown in \cite{LiJi2020} when $\varepsilon$ is sufficiently small, namely, $\varepsilon \|\gamma\|_{L^\infty(0,\infty)}<1$. \Cref{theo26} relaxes this condition at the expense of the algebraic growth condition~\eqref{phi1}. Let us recall that existence of weak solutions is also obtained in \cite{TaWi2017} when $\min\{\gamma\}>0$ and in \cite{DKTY2019,YoKi2017} when $\gamma(z)=1/(c+z^k)$ for $c\ge 0$ and $k>0$ sufficiently small, these conditions being removed in \Cref{theo26} as well. The constraint on $\varepsilon$ required in~\cite{LiJi2020} and the algebraic growth~\eqref{phi1} assumed in \Cref{theo26} are likely to be of a technical nature, as global existence of weak solutions is proved in \cite{BLT2020} in the particular case $\gamma(z)=e^{-z}$. As for classical solutions, well-posedness in that setting is shown in \cite{FuJi2020b, FuJi2020c, FuSe2021, JLZ2022, JW2020, LyWa2021}, provided that $\gamma'\le 0$ with $\gamma(z)\to 0$ as $z\to\infty$.

\begin{remark} 
If $v^{in}$ is bounded below by a strictly positive constant, then it is easy to see that the same will remain true for $v(t, \cdot)$, with a constant that decays exponentially fast with time. Then the behavior of $\gamma$ at zero is irrelevant and one can relax the assumption that $\gamma$ is continuous on $[0,\infty)$ and replace it with continuity on $(0, \infty)$ only. This is true also for the next theorems of existence.
\end{remark}

\begin{remark}
Since $u$ lies in $L^p((0,T)\times\Omega)$ for $p\in (1,2)$, $v$ is actually a strong solution to \eqref{ivp2}: each term lies in $L^p((0,T)\times\Omega)$ for all $T>0$ and the equation holds almost everywhere on $(0,\infty)\times\Omega$. One can furthermore show that the formula \eqref{ivp1K} holds in a strong sense: thanks to the boundedness of $\gamma$, each term lies in $L^2((0,T)\times\Omega)$ for all $T>0$ and the equation holds almost everywhere on $(0,\infty)\times\Omega$. 
\end{remark}

We next turn to the large time behavior of solutions to~\eqref{ivp}. While a complete description of the dynamics for an arbitrary motility function $\gamma$ seems to be out of reach, it is shown in \cite{AhnYoon2019} that, when $\varepsilon=0$ and $\gamma(z)=z^{-k}$ for some $k\in (0,1]$, solutions converge to spatially homogeneous steady states. A key ingredient in their proof is  the construction of a Lyapunov functional but this property breaks down when $\varepsilon>0$. Nevertheless, we are able to prove that the system~\eqref{ivp} admits a Lyapunov functional, which is different from that constructed in~\cite{AhnYoon2019} but applies to the same call of motilities, and requires actually extra conditions on the monotonicity of $\gamma$, as described now.

\begin{theorem}\label{thm.lt}
Let $\Omega$ be a smooth bounded domain of $\RR^N$, with $N\ge 2$. Assume that $\varepsilon>0$, that  $\gamma\in C([0,\infty))\cap C^3((0,\infty))$  is positive, and that
\begin{equation}\label{mon}
\gamma' \le 0, \qquad  (z \mapsto z\,\gamma(z))' \ge 0\,. 
\end{equation}
Consider nonnegative initial conditions $(u^{in},v^{in})\in W^{1,r}(\Omega;\mathbb{R}^2)$ for some $r>N$ and denote the corresponding  global classical solution to~\eqref{ivp} by $(u,v)$ \cite{JLZ2022, FuSe2021}. Setting $m := \langle u^{in}\rangle$, we define $G_0\in C^1([0,\infty))\cap C^4((0,\infty))$ by 
	\begin{equation}
	G_0'(z) := 2 z \gamma(z) - m \gamma(z) - m\gamma(m)\,, \quad z\ge 0\,, \qquad G_0(m)=0\,. \label{G0}
	\end{equation}
	Then $G_0$ is nonnegative and convex on $(0,\infty)$, and 
	\begin{equation}
	\frac{\dd}{\dd t} \mathcal{L}_0(u(t),v(t)) + \mathcal{D}_0(u(t),v(t)) = 0\,, \qquad t>0\,, \label{lf0} 
	\end{equation}
	where
	\begin{equation}
	\mathcal{L}_0(u,v) := \frac{1}{2} \|\nabla\mathcal{K}(u-m)\|_2^2 + \varepsilon \int_\Omega G_0(v)\ \mathrm{d}x \ge 0, \label{L0}
	\end{equation}
	and
	\begin{equation}
	\begin{split}
	\mathcal{D}_0(u,v) & := \int_\Omega G_0''(v) |\nabla v|^2\ \mathrm{d}x + \int_\Omega (u-v)^2 \gamma(v)\ \mathrm{d}x \\
	& \qquad + \int_\Omega (v-m)(v\gamma(v)-m\gamma(m))\ \mathrm{d}x \ge 0\,.
	\end{split} \label{D0}
	\end{equation}
Moreover, 
\begin{equation}
	\sup_{t\ge 0} \left\{ \mathcal{L}_0(u(t),v(t)) + \|v(t)\|_{H^1}^2 \right\} + \int_0^\infty \left[ \mathcal{D}_0(u(s),v(s)) + \varepsilon \|\partial_t v(s)\|_2^2 \right]\ \mathrm{d}s < \infty\,, \label{lt.00}
\end{equation}
and
\begin{equation*}
	\lim_{t\to\infty} \left\{ \|\mathcal{K}(u(t) - m)\|_2 + \|v(t) - m\|_2 \right\} = 0\,.
\end{equation*}
\end{theorem}

The construction of the Lyapunov functional $\mathcal{L}_0$ and its consequence with respect to the long-term behavior are actually the main contribution of Theorem~\ref{thm.lt}, the existence and uniqueness of a global classical solution to~\eqref{ivp} being granted by \cite[Theorem~1.1]{FuSe2021} and \cite[Remark~1.5]{JLZ2022}. As in \cite{AhnYoon2019} which is devoted to the parabolic-elliptic version of \eqref{ivp} corresponding to $\varepsilon=0$, Theorem~\ref{thm.lt} applies to $\gamma(z)=z^{-k}$, $k\in (0,1]$, and we have thus constructed a Lyapunov functional in that case. A side remark is that pattern formation is excluded by Theorem~\ref{thm.lt}, which is consistent with the outcome of \cite{2012PRL}, where the formation of stripes is observed for a motility $\gamma$ with a very fast decay at infinity.

The following observation on the existence of nonconstant steady states
indicates that the choice $k=1$ in fact even corresponds to a critical nonlinearity
in the family of such algebraic motility rates:
\begin{proposition}\label{prop99}
  Let $N\ge 2$, $k\in (1,\frac{N+2}{(N-2)_+})$ and $\gamma(z):=z^{-k}$ for $z>0$.
  Then given any smooth bounded domain $\Omega_0\subset\R^N$, one can find $R_0>0$ such that whenever $R>R_0$, defining $\Omega:=R\Omega_0=\{Rx: x\in\Omega_0\}$, there are positive nonconstant functions $u\in C^2(\bom)$ and $v\in C^2(\bom)$ satisfying 
  \bas
	\left\{ \begin{array}{l}
	0 = \Delta \big(u\gamma(v)\big)
	\qquad \mbox{in } \Omega, \\[1mm]
	0 = \Delta v - v + u
	\qquad \mbox{in } \Omega, \\[1mm]
	0 = \nabla u\cdot \mathbf{n} = \nabla v\cdot \mathbf{n}
	\qquad \mbox{on } \pO.
	\end{array} \right.
  \eas
\end{proposition}
Proposition \ref{prop99} holds also in space dimension one $N=1$. We refer to \cite{WangXu21} for a more complete description of steady states in that case.

\bigskip

In the final part of this manuscript, we aim at making sure that in the presence of additional zero-order dissipative
mechanisms in the flavor of logistic-type source and degradation terms, global solutions can be constructed
actually without any substantial restriction on the strength of degeneracies in cell diffusion at large values of the
signal. We work in a framework somewhat less relaxed than that considered above. Typically, only one integration by parts is performed, so that the solutions considered
here are sometimes called ``weak solutions'' instead of ``very weak solutions''.
\begin{definition}\label{dw3}
  Let  $\Omega$ be a smooth bounded domain of $\RR^N$, with $N\ge2$, and $\varepsilon>0$, $\gamma\in C((0,\infty))$ and $h\in C([0,\infty))$, and suppose that
  $u^{in}\in L^1(\Omega)$ and $v^{in}\in L^1(\Omega)$ are nonnegative. We then call a pair $(u,v)$ of functions such that, for all $T>0$,
  \be{w31}
	u\in L^1((0,T)\times\Omega) \qquad \mbox{and} \qquad v\in L^1((0,T);W^{1,1}(\Omega)),
  \ee
  a {\em global weak solution of (\ref{02})} if
  $u\ge 0$ and $v>0$ a.e.~in $(0,\infty)\times\Omega$, if, for all $T>0$,
  \be{w32}
	u\,h(u) \in L^1((0,T)\times\Omega)
	\qquad \mbox{and} \qquad
	u\gamma(v)\in L^1((0,T);W^{1,1}(\Omega)),
  \ee
  and if
  \be{wu3}
	- \int_0^\infty \io u\partial_t \varphi \dd x \dd t - \io u^{in} \varphi(0) \dd x
	= \int_0^\infty \io \nabla \big\{ u\gamma(v) \big\} \cdot \nabla \varphi \dd x \dd t + \int_0^\infty \io u\,h(u) 
		\dd x \dd t 
  \ee
  as well as
  \bea{wv3}
	- \int_0^\infty \io \varepsilon v\partial_t\varphi \dd x \dd t - \io \varepsilon v^{in}\varphi(0) \dd x 
	&=& - \int_0^\infty \io  \nabla v\cdot\nabla\varphi \dd x \dd t
	+ \int_0^\infty \io (-v+u) \varphi \dd x \dd t 
  \eea
 hold for any $\varphi\in C_0^\infty([0,\infty)\times\bom)$.
\end{definition}
Our analysis in this direction will be based on a strategy quite independent from that pursued in the previous parts,
focusing on the detection of entropy-like features enjoyed by functionals of the form 
$$\io \left[ u\ln (u+e) + \eps |\nabla v|^2 \right] \dd x.$$
Accordingly, in its most straightforward version detailed in Lemma~\ref{lem222}, this approach will presuppose regularity properties
of $u^{in}$ and especially of $v^{in}$ which go somewhat beyond those from (\ref{init}). 
In view of our principal intention described above, we will refrain from scrutinizing minimal requirements on initial regularity,
and rather formulate our main result in this respect in the following form conveniently accessible to a fairly compact analysis:
\newcommand{\Om}{\Omega}
\newcommand{\ueta}{u_\eta}
\newcommand{\veta}{v_\eta}
\begin{theorem}\label{theo262b}
  Let $N\ge 2$ and $\Om\subset\R^N$ be a bounded domain with smooth boundary, and suppose that $\eps>0$, that
  \be{gamma2}
	\gamma\in C^3((0,\infty))
	\mbox{ is such that $\gamma>0$ in $[0,\infty)$ and } \quad
	\sup_{s>s_0} \left \{ \gamma(s) + \frac{s\gamma'^2(s)}{\gamma(s)} \right\} <\infty
	\quad \mbox{for all } s_0>0,
  \ee
  and that $h\in C([0,\infty))$ satisfies
  \be{h}
	\lim_{s\to\infty} \frac{h(s)\ln s}{s} = -\infty\,.
  \ee
  Then for any choice of $u^{in} \in C(\bom)$ and $v^{in}\in W^{1,\infty}(\Om)$ such that
  $u^{in}\ge 0$ and $v^{in}>0$ in $\bom$,
  the problem (\ref{02}) possesses at least one global weak solution in the sense of Definition \ref{dw3}.
  This solution has the additional properties that, for all $T>0$,
  \be{262.1}
	\left\{ \begin{array}{l}
	u\in L^\infty((0,T);L\log L(\Om)) \cap L^2((0,T);L^2(\Om))
	\qquad \mbox{ and }  \\[1mm]
	v\in L^\infty((0,T);H^1(\Om)) \cap L^2((0,T);H^2(\Om))
	\end{array} \right.
  \ee
  as well as
  \be{262.11}
	u\sqrt{\gamma(v)}\in L^{4/3}((0,T);W^{1,4/3}(\Om)).
  \ee
\end{theorem}
The existence of classical bounded solutions to~\eqref{02} is proven for growth functions $h(z)=h_0 \, (1-z^l)$ when $l=1$ and $h_0$ is large enough in \cite{LiuXu19, WW2019}, and when $h_0>0$ and $l>\max(2,(N+2)/2)$ in \cite{lv_wang_2020} (under additional assumptions on $\gamma$). The case where $l=1$ and $h_0>0$ is treated in two-dimensional domains in \cite{JKW2018}. Here, we obtain weak solutions for a growth function of this form with $l\ge1$ and $h_0>0$.

The structure of the paper is the following: \Cref{sec.pmt} is devoted to the proof of \Cref{theo26}. Then \Cref{thm.lt} and \Cref{prop99} are proven in \Cref{sec.ltblf}. Finally, the results on system (\ref{02}) are discussed in Section~\ref{sec.ltg}. 

\section{Existence in the absence of logistic-type growth}\label{sec.pmt}

Let $\psi \in C^\infty_0(\RR)$ such that $\psi\ge0$, $\supp \psi \subset (-1,1)$, $\|\psi\|_{L^1(\RR)}=1$, and let, for $\eta>0$,
\begin{equation*} 
	\psi_\eta(z) := \frac{1}{\eta} \psi\left(\frac{z}{\eta}\right)\,, \qquad z\in [0,\infty)\,.
\end{equation*}
For $\eta\in (0,1)$, we define
\begin{equation*} 
	\tge(z) := \eta + \left(\psi_{\eta} \ast \gamma\right)(z+\eta)\,, \qquad z\in [0,\infty)\,, 
\end{equation*}
where the symbol $\ast$ indicates the convolution product on $\RR$ ($\gamma$ being extended on $\RR$ by symmetry). We first obtain some properties of $\tge$.
\begin{lemma}\label{lem.gameta}
For all $\eta\in(0,1)$, the function $\tge$ satisfies
\begin{eqnarray}
0 < \eta \le \tge (z) \le K_0 + 1 \,, \qquad z\in [0,\infty)\,, \label{bdeta}\\
\frac{1}{\tge(z)} \le 3^k K_1  (z+1)^k  \,, \qquad z\in [0,\infty)\,, \label{phi1eta}
\end{eqnarray}
where $K_0 = ||\gamma||_{L^\infty(0,\infty)}$ and $K_1$ and $k$ are defined in \eqref{phi1}.
\end{lemma}
\begin{proof}
We first see that thanks to the nonnegativity of $\gamma$ and $\psi$, we have $\tge \ge \eta$ and
\begin{equation*}
\|\tge \|_{L^\infty(0,\infty)} \le \|\gamma\|_{L^\infty(0,\infty)} \|\psi_{\eta}\|_{L^1(\RR)} + \eta = K_0 + \eta \le K_0 + 1 \,.
\end{equation*}
Then, for $z>0$, we compute, using \eqref{phi1}, 
\begin{align*}
\tge (z)
&= \eta + \int_{-\eta}^{\eta} \psi_{\eta}(\tilde z) \gamma(z+\eta-\tilde z)\, d\tilde z  \ge \eta + \int_{-\eta}^{\eta}  \frac{\psi_{\eta}(\tilde z)}{K_1(1+z+\eta-\tilde z)^k}\, d\tilde z\\
&\ge \eta + \int_{-\eta}^{\eta}  \frac{\psi_{\eta}(\tilde z)}{K_1(1+z+2\eta)^k}\, d\tilde z \ge \eta + \frac{1}{K_1(1+2 \eta)^k (1+z)^k} \\
& \ge \frac{1}{3^k K_1(1+z)^k} \,,
\end{align*}
and the proof is complete.
\end{proof}

Next,  let $(u_\eta^{in},v_\eta^{in})_\eta$ be nonnegative functions in $C(\bar{\Omega})\times W^{1,\infty}(\Omega)$ such that
\begin{equation}
	\langle u_\eta^{in} \rangle = \langle u^{in} \rangle =: m\,, \qquad \langle v_\eta^{in} \rangle = \langle v^{in} \rangle\,, \label{a1}
\end{equation}
and
\begin{equation}
		\begin{split}
	& \lim_{\eta\to 0} \left\{ \|u_\eta^{in} - u^{in}\|_{(H^1)'} + \|u_\eta^{in} - u^{in}\|_{p_0} + \| v_\eta^{in} - v^{in}\|_2 \right\} = 0\,, \\
	& \sup_{\eta} \|u_\eta^{in}\|_{p_0} \le 1 + \|u^{in}\|_{p_0}\,, \qquad \sup_{\eta} \|v_\eta^{in}\|_\infty \le 1 + \|v^{in}\|_\infty\,.
	\end{split} \label{a2}
\end{equation}
Moreover, if $v^{in}\in W^{1,q}(\Omega)$ for some $q\in (1,2)$, then $(v_\eta^{in})_\eta$ can be constructed so as to satisfy
\begin{equation}
	\sup_\eta\|v_\eta^{in}\|_{W^{1,q}} \le 1 + \|v^{in}\|_{W^{1,q}}\,. \label{a2b}
\end{equation}
Thanks to the regularity of $\tge$, $u_\eta^{in}$, and $v_\eta^{in}$, we are in a position to apply \cite[Theorem~1.2]{TaWi2017} to obtain the existence of a nonnegative global weak solution $(u_\eta,v_\eta)$ to the initial value problem  
\begin{subequations}\label{rivp}
	\begin{align}
		\partial_t u_\eta - \Delta (u_\eta \tge(v_\eta)) = 0\,,& \qquad (t,x)\in (0,\infty)\times\Omega\,, \label{rivp1} \\
		\varepsilon \partial_t v_\eta  - \Delta v_\eta + v_\eta = u_\eta\,,& \qquad (t,x)\in (0,\infty)\times\Omega\,, \label{rivp2} \\	
		\nabla (u_\eta \tge(v_\eta)) \cdot \mathbf{n}  = \nabla v_\eta \cdot \mathbf{n} = 0\,,& \qquad (t,x)\in (0,\infty)\times\partial\Omega\,, \label{rivp3} \\
		(u_\eta,v_\eta)(0, \cdot)  = (u_\eta^{in},v_\eta^{in})\,,& \qquad x\in \Omega\,, \label{rivp4}
\end{align}
\end{subequations}
satisfying
\begin{align}
u_\eta & \in L^2((0,T)\times\Omega) \cap L^{(N+2)/(N+1)}((0,T);W^{1,(N+2)/(N+1)}(\Omega))\,, \label{a10}\\
v_\eta & \in L^\infty((0,T);H^1(\Omega))\cap L^2((0,T);H^2(\Omega))\,, \label{a11}
\end{align} 
for any $T>0$.
\medskip

As a first consequence of~\eqref{a1} and~\eqref{rivp}, we identify the time evolution of the space averages of $u_\eta$ and $v_\eta$.

\begin{lemma}\label{lem.a2}
For $t\ge 0$, 
\begin{equation}
	\langle u_\eta(t) \rangle = m\,, \qquad \langle v_\eta(t) \rangle = \langle v^{in} \rangle e^{-t/\varepsilon} + m(1-e^{-t/\varepsilon}) \le \max\{ \langle v^{in} \rangle , m \}\,. \label{a3}
\end{equation}
\end{lemma}

\begin{proof}
	The first identity in \eqref{a3} readily follows from \eqref{a1}, \eqref{rivp1}, and \eqref{rivp3}, after integration over $\Omega$. We next integrate \eqref{rivp2} over $\Omega$ and use \eqref{a1} and \eqref{rivp3} to obtain
	\begin{equation}
		\varepsilon \frac{\dd }{\dd t} \|v_\eta\|_1 + \|v_\eta\|_1 = \|u_\eta\|_1 = m |\Omega|\,, \qquad t\ge 0\,. \label{a3b}
	\end{equation}
	Integrating \eqref{a3b} completes the proof of Lemma \ref{lem.a2}.
\end{proof}

\bigskip

 We define then $w_\eta = w_\eta(t,x)$ as the unique nonnegative solution of the elliptic equation  (in the $x$ variable, for a given $t$)
\begin{subequations}\label{no4}
	\begin{align}
		 - \Delta w_\eta + w_\eta & = u_\eta\,, \qquad (t,x)\in (0,\infty)\times\Omega\,, \label{no3a} \\	
		\nabla w_\eta \cdot \mathbf{n}    & = 0\,, \qquad (t,x)\in (0,\infty)\times\partial\Omega\,. \label{no3b} 
\end{align}
\end{subequations}
Thanks to this auxiliary problem, we have the following lemma.

\begin{lemma}\label{lem.a21}
For $t\in [0,T]$, there exists some constant $C(T) \ge 0$ such that
\begin{equation}\label{a21}
\|u_\eta(t)\|^2_{(H^1)'} + ||v_\eta(t) ||^2_2 + \int_0^t  ||v_\eta(s) ||^2_{H^1} \dd s 
+  \int_0^t  \int_{\Omega}  u_\eta(s, \cdot)^2\, \gamma_{\eta} (v_\eta(s, \cdot)) \dd x \dd s \le C(T).
\end{equation}
\end{lemma}

\begin{proof}
 We observe that, by \eqref{bdeta}, \eqref{rivp1}, and \eqref{no4},  
\begin{eqnarray} \frac12 \frac{\dd}{\dd t} ||w_\eta||^2_{H^1} &=& \int_{\Omega} w_\eta\, \Delta ( u_\eta\, \gamma_{\eta} (v_\eta)) \dd x = \int_{\Omega}  u_\eta\, \gamma_{\eta} (v_\eta)\, (w_\eta - u_\eta) \dd x \nn \\
&\le& (K_0+1) \, ||w_\eta||^2_{H^1} -  \int_{\Omega}  u_\eta^2\, \gamma_{\eta} (v_\eta) \dd x \,. \label{bb}
 \end{eqnarray}
Also, it follows from~\eqref{rivp2} that
\begin{eqnarray} \frac{\var}2 \, \frac{\dd}{\dd t} ||v_\eta||^2_{2} +  ||v_\eta||^2_{H^1} & = &  \int_{\Omega}  u_\eta \, v_\eta \dd x = \int_\Omega \big( v_\eta w_\eta + \nabla v_\eta\cdot \nabla w_\eta \big) \dd x\nn \\
\label{bb2}
 &\le& \frac12 ||v_\eta||^2_{H^1} +  \frac12 ||w_\eta||^2_{H^1} \,,
\end{eqnarray}
so that 
\begin{equation}\label{bb3}
 \frac{\dd}{\dd t} \left[  \frac12 ||w_\eta||^2_{H^1} +  \var\, ||v_\eta||^2_{2} \right] +  ||v_\eta||^2_{H^1} 
+  \int_{\Omega}  u_\eta^2\, \gamma_{\eta} (v_\eta) \dd x
\le (2+ K_0) \,  ||w_\eta||^2_{H^1}\,.
\end{equation}
After integration with respect to time, for all $T>0$, there exists some constant $C(T) \ge 0$ (we emphasize the dependence with respect to $T$, but it also depends on the parameters of the problem but not on $\eta\in (0,1)$) such that
\begin{equation*}
 ||w_\eta(t)||^2_{H^1} +   ||v_\eta(t) ||^2_{2} + \int_0^t  ||v_\eta(s) ||^2_{H^1} \dd s 
+  \int_0^t  \int_{\Omega}  u_\eta(s, \cdot)^2\, \gamma_{\eta} (v_\eta(s, \cdot)) \dd x \dd s \le C(T)\,, \qquad t\in [0,T]\,.
\end{equation*}
Recalling the definition of $\mathcal{A}^{-1}$ in \eqref{opa}, we note that $w_\eta = \mathcal{A}^{-1} u_\eta$, and conclude the proof of the lemma by equivalence of the norms \eqref{H1pnorm} and \eqref{H1pnormA}.
\end{proof}

Building upon \eqref{a21}, we derive additional estimates on $(v_\eta)_\eta$ with the help of a comparison argument introduced in \cite{FuJi2020} (and subsequently developed further in \cite{FuJi2020b, FuJi2020c, FuSe2021, JLZ2022, LiJi2020, LyWa2021, LyWa2021a}) and parabolic maximal regularity. 
	
\begin{lemma}\label{lem.a22}
For all $q\in (1,\infty)$ and $T>0$, there exists some constant $C(T,q) \ge 0$ such that
\begin{equation}\label{bb5}
 \int_0^T   ||v_\eta(t) ||_{q}^q \dd t \le C(T,q). 
\end{equation}
\end{lemma}

\begin{proof}
We start with the comparison argument introduced in \cite{FuJi2020} and deduce from \eqref{bdeta}, \eqref{rivp}, \eqref{no4}, and the definition~\eqref{opa} of $\mathcal{A}$ that, for $t\ge 0$,
\begin{equation*}
	\mathcal{A}\big( \partial_t w_\eta(t) + u_\eta(t)\tge(v_\eta(t)) \big) = u_\eta(t) \tge(v_\eta(t)) \le (1+K_0) u_\eta(t) = \mathcal{A}\big[ (1+K_0) w_\eta(t) \big] \;\;\text{ in }\;\; \Omega\,,
\end{equation*}
with $\nabla \big( \partial_t w_\eta(t) + u_\eta(t)\tge(v_\eta(t)) - (1+K_0) w_\eta(t) \big)\cdot \mathbf{n} = 0$ on $\partial\Omega$. We then infer from the (elliptic) comparison principle that
\begin{equation*}
	\partial_t w_\eta + u_\eta\tge(v_\eta) \le (1+K_0) w_\eta \;\;\text{ in }\;\; [0,\infty)\times\Omega\,.
\end{equation*}
Hence, since $u_\eta$ and $\tge$ are nonnegative, 
\begin{equation*}
	0 \le w_\eta(t,x) \le e^{(1+K_0)t} w_\eta(0,x)\,, \qquad (t,x)\in [0,\infty)\times \Omega\,.
\end{equation*}
In addition, recalling that $p_0>N/2$, the continuous embedding of $W^{2,p_0}$ in $L^\infty(\Omega)$, elliptic regularity, and \eqref{a2} imply that
\begin{equation*}
	\|w_\eta(0)\|_\infty \le C \|\mathcal{A}^{-1}(u_\eta^{in})\|_{W^{2,p_0}} \le C \|u_\eta^{in}\|_{p_0} \le C\,.
\end{equation*}
Combining the above two estimates leads us to
\begin{equation}
	\sup_{t\in [0,T]} \|w_\eta(t)\|_\infty\le C(T) \label{mk1}
\end{equation}
for each $T>0$.

We next define $z_\eta := \mathcal{A}^{-1}v_\eta$ and deduce from \eqref{rivp} that $z_\eta$ solves
\begin{equation}\label{mk2}
	\begin{split}
		\varepsilon \partial_t z_\eta  - \Delta z_\eta + z_\eta = w_\eta\,,& \qquad (t,x)\in (0,\infty)\times\Omega\,, \\	
		\nabla z_\eta \cdot \mathbf{n} = 0\,,& \qquad (t,x)\in (0,\infty)\times\partial\Omega\,,  \\
		z_\eta(0, \cdot)  = z_\eta^{in} := \mathcal{A}^{-1}v_\eta^{in}\,,& \qquad x\in \Omega\,, 
	\end{split}
\end{equation}
Now, let $q\in (1,\infty)$. We recall that the operator $\mathcal{A}_\varepsilon$ defined by
\begin{equation}
\begin{split}
	D(\mathcal{A}_\varepsilon) & := \big\{ \phi\in W^{2,q}(\Omega)\ :\ \nabla\phi\cdot \mathbf{n} = 0 \;\text{ on }\; \partial\Omega \big\}\,, \\
	\mathcal{A}_\varepsilon \phi & := \frac{1}{\varepsilon} \big( - \Delta \phi + \phi \big)\,, \qquad \phi\in D(\mathcal{A}_\varepsilon)\,,
\end{split} \label{Aeps}
\end{equation}
generates an analytic semigroup  $\left( e^{-t\mathcal{A}_\varepsilon} \right)_{t\ge 0}$ of contractions on $L^q(\Omega)$ \cite[Theorem~7.3.5]{Pazy1983} (note that $\mathcal{A}_1 = \mathcal{A}$, see~\eqref{OpA}). With this notation, a representation formula  for $z_\eta$ can be derived from \eqref{mk2} which reads
\begin{equation}
	\varepsilon z_\eta(t) = \varepsilon e^{-t\mathcal{A}_\varepsilon} z_\eta^{in} + \int_0^t e^{-(t-s)\mathcal{A}_\varepsilon} w_\eta(s)\dd s\,, \qquad t\ge 0\, . \label{mk10}
\end{equation}
On the one hand, we infer from \eqref{mk1} and \cite[Th\'eor\`eme~1]{Lamb1987} that
\begin{equation*}
	\left\| t \mapsto \int_0^t e^{-(t-s)\mathcal{A}_\varepsilon} w_\eta(s)\dd s \right\|_{L^q((0,T);W^{2,q}(\Omega))} \le C(q) \| w_\eta \|_{L^q((0,T)\times \Omega)} \le C(T,q)\,.
\end{equation*}
On the other hand, classical properties of semigroups, elliptic regularity, and \eqref{a2} entail that
\begin{equation*}
	\left\| e^{-t\mathcal{A}_\varepsilon} z_\eta^{in}\right\|_{L^q((0,T);W^{2,q}(\Omega))} \le C(T,q) \|z_\eta^{in}\|_{W^{2,q}} \le C(T,q) \|v_\eta^{in}\|_q \le C(T,q)\,.
\end{equation*}
Consequently,
\begin{equation*}
	\varepsilon \left\| z_\eta \right\|_{L^q((0,T);W^{2,q}(\Omega))} \le C(T,q)\,,
\end{equation*}
an estimate which completes the proof since $v_\eta=\mathcal{A}z_\eta$.
\end{proof}

We next turn to $(u_\eta)_\eta$ and draw the following consequence of Lemma~\ref{lem.a21} and Lemma~\ref{lem.a22}.

\begin{lemma}\label{lem.a23}
For all $p\in (1,2)$ and $T>0$, there exists some constant $C(T,p) \ge 0$ such that
\begin{equation}\label{GNdidthat2}
\int_0^T\|u_\eta(t)\|_p^p \dd t   \le C(T,p)\,.
\end{equation}
\end{lemma}

\begin{proof}
By \eqref{phi1eta} and H\"older's inequality,
\begin{align*}
	\|u_\eta\|_p^p & = \int_{\Omega} \left( u_\eta\, \sqrt{\tge(v_\eta)}  \right)^p \, \tge(v_\eta)^{-p/2} \dd x  \\
	& \le  \bigg( \int_{\Omega}  u_\eta^2\, \tge(v_\eta) \dd x \bigg)^{p/2}\,  \bigg( \int_{\Omega}  \tge(v_\eta)^{ - p/(2-p) } \dd x \bigg)^{(2-p)/2} \\
	& \le 3^k K_1 \left\| u_\eta \sqrt{\tge(v_\eta)} \right\|_2^{p} \left[ \int_\Omega (1+v_\eta)^{pk/(2-p)} \dd x \right]^{(2-p)/2} \\
	& \le C(p) \left\| u_\eta \sqrt{\tge(v_\eta)} \right\|_2^{p} \left( 1 + \| v_\eta\|_{pk/(2-p)}^{pk/2} \right)\,.
\end{align*}
Integrating the above inequality with respect to time over $(0,T)$ and using H\"older's inequality gives
\begin{align*}
	\int_0^T \|u_\eta\|_p^p \dd t & \le C(p) \int_0^T \left\| u_\eta \sqrt{\tge(v_\eta)} \right\|_2^p \left( 1 + \| v_\eta\|_{pk/(2-p)}^{pk/2} \right)\dd t \\
	& \le C(p) \left( \int_0^T \left\| u_\eta \sqrt{\tge(v_\eta)} \right\|_2^{2} \dd t \right)^{p/2} \left( \int_0^T \left( 1 + \| v_\eta\|_{pk/(2-p)}^{pk/2} \right)^{\frac{2}{(2-p)}}\dd t \right)^{(2-p)/2} \\
	& \le C(p) \left( \int_0^T \left\| u_\eta \sqrt{\tge(v_\eta)} \right\|_2^{2} \dd t \right)^{p/2} \left( \int_0^T \left( 1 + \| v_\eta\|_{pk/(2-p)}^{pk/(2-p)} \right)\dd t \right)^{(2-p)/2}\,.
\end{align*}
Lemma~\ref{lem.a23} then readily follows from the above estimate due to \eqref{a21} and Lemma~\ref{lem.a22} (with $q=pk/(2-p)$).
\end{proof}

Exploiting the outcome of Lemma~\ref{lem.a23} provides additional estimates on $(v_\eta)_\eta$.
	
\begin{lemma}\label{lem.a26}
	Let $q\in (1,2)$ and $T>0$ and assume that $v^{in}\in W^{1,q}(\Omega)$. Then there exists some constant $C(T,q)\ge 0$ such that
	\begin{equation*}
		\sup_{t\in [0,T]} \|v_\eta(t)\|_{W^{1,q}} \le C(T,q)\,.
	\end{equation*}
\end{lemma}

\begin{proof}
Recalling the notation introduced in \eqref{Aeps}, it follows from \eqref{rivp} that 
\begin{equation*}
	\varepsilon v_\eta(t) = \varepsilon e^{-t\mathcal{A}_\varepsilon} v_\eta^{in} + \int_0^t e^{-(t-s)\mathcal{A}_\varepsilon} u_\eta(s) \dd s\,, \qquad t\ge 0\,.
\end{equation*}	
Using classical properties of the semigroup $(e^{-t\mathcal{A}_\varepsilon)})_{t\ge 0}$, see \cite[V.Theorem~2.1.3]{Aman1995} for instance, along with \eqref{a2b}, we obtain
\begin{equation}
	\varepsilon \|v_\eta(t)\|_{W^{1,q}} \le C(q) + C(q) \int_0^t (t-s)^{-1/2} \|u_\eta(s)\|_q \dd s\,, \qquad t\ge 0\,. \label{mk5}
\end{equation}
	
We next fix $p\in (q,2)$ and set
\begin{equation*}
	\omega := 1 + \frac{q(p-1)}{p(q-1)} \in \left( 2 , \frac{2q(p-1)}{p(q-1)} \right)\,.
\end{equation*}
We infer from \eqref{a1} and H\"older's inequality that
\begin{equation*}
	\|u_\eta\|_q \le \|u_\eta\|_p^{\frac{p(q-1)}{q(p-1)}} \|u_\eta\|_1^{\frac{p-q}{q(p-1)}} \le C(p,q) \|u_\eta\|_p^{\frac{1}{\omega-1}}\,.
\end{equation*} 
Together with H\"older's inequality, the above inequality ensures that, for $t>0$,
\begin{align*}
	\int_0^t (t-s)^{-1/2} \|u_\eta(s)\|_q \dd s & \le \left( \int_0^t (t-s)^{-\frac{\omega}{2(\omega-1)}} \dd s \right)^{(\omega-1)/\omega} \left( \int_0^t \|u_\eta(s)\|_q^\omega \dd s \right)^{1/\omega} \\
	& \le C(p,q) t^{\frac{\omega-2}{2(\omega-1)}} \left( \int_0^t \|u_\eta(s)\|_p^{\frac{\omega}{\omega-1}} \dd s \right)^{1/\omega}\,.
\end{align*}
Since both $\omega/(\omega-1)$ and $p$ lie in $(1,2)$, we deduce from Lemma~\ref{lem.a23} that, for $T>0$ and $t\in [0,T]$, 
\begin{equation*}
	\int_0^t (t-s)^{-1/2} \|u_\eta(s)\|_q \dd s \le C(T,p,q)\,.
\end{equation*}
Inserting the above estimate in \eqref{mk5} completes the proof.
\end{proof}

 \subsection{Compactness and convergence}
 We now collect the estimates that are uniform with respect to $\eta\in (0,1)$, which will prove useful when passing to the limit $\eta \rightarrow 0$.

\begin{proposition}\label{prop.un0}
Let $T>0$ , $q\in (1,\infty)$, and $p\in (2N/(N+2),2)$. There are $C_0(T)>0$, $C_1(T,q)>0$, and $C_2(T,p)>0$ such that
\begin{subequations}\label{es.un0}
\begin{eqnarray}
\langle u_\eta(t) \rangle = m\,, \qquad 0\le \langle v_\eta(t) \rangle \le \max\{ \langle v^{in} \rangle , m \}\,,&  \qquad t\in [0,T]\,, \label{es.un01} \\
\|\nabla \mathcal{K}(u_\eta(t)-m)\|_2^2 \le 2 C_0(T)\,,& \qquad t\in [0,T]\,,\label{es.un02} \\
\|v_\eta(t)\|_{2}^2 \le C_0(T)\,,&  \qquad t\in [0,T]\, , \label{es.un04b}\\
\int_0^T \|v_\eta(s)\|_{H^1}^2 \ \mathrm{d}s \le C_0(T)\,,& \label{es.un04}\\
\int_0^T \left\|(u_\eta(s)-v_\eta(s)) \sqrt{\gamma_\eta (v_\eta(s))} \right\|_2^2 \ \mathrm{d}s \le C_0(T)\,,& \label{es.un03}\\
\int_0^T \left\| u_\eta(s) \sqrt{\gamma_\eta(v_\eta(s))} \right\|_2^2 \ \mathrm{d}s \le C_0(T)\,,& \label{es.un06a}\\
\int_0^T \|\partial_t \mathcal{K} (u_\eta(s)-m) \|_2^2 \ \mathrm{d}s \le C_0(T)\,,& \label{es.un06b}\\
\int_0^T \|v_\eta(s)\|_q^q \dd s \le C_1(T,q)\,, & \label{mk6} \\
\int_0^T \|u_\eta(s)\|_p^p\ \mathrm{d}s \le C_2(T,p)\,,& \label{es.un07}	\\
\int_0^T \|\partial_t v_\eta(s)\|_{(H^1)'}^p\ \mathrm{d}s \le C_2(T,p)\,.& \label{es.un05} 
\end{eqnarray}
\end{subequations}
\end{proposition}

\begin{proof}

The estimate \eqref{es.un01} is given by Lemma~\ref{lem.a2}, the estimates~\eqref{es.un04b}, \eqref{es.un04} and~\eqref{es.un06a} are given by Lemma~\ref{lem.a21} and the estimates~\eqref{mk6} and~\eqref{es.un07} are given by Lemma~\ref{lem.a22} and Lemma~\ref{lem.a23}, respectively. Estimate~\eqref{es.un02} is a consequence of Lemma~\ref{lem.a21}, the mass estimate~\eqref{es.un01} and the definition of the $(H^1)'$ norm in \eqref{H1pnorm}. Using again~\eqref{H1pnorm} and~\eqref{a1}, one can deduce from \eqref{rivp1} and \eqref{rivp3} that
\begin{equation}
	\partial_t \mathcal{K}(u_\eta-m) + u_\eta \gamma_\eta(v_\eta)  = \langle u_\eta \gamma_\eta(v_\eta) \rangle\,, \qquad (t,x)\in (0,\infty)\times\Omega\,, \label{kivp0}
\end{equation}
so that estimate~\eqref{es.un06b} is a consequence of \eqref{es.un06a} and the (uniform) upper bound \eqref{bdeta} on $\gamma_\eta$. It next follows from \eqref{rivp2}, H\"older's inequality, and the continuous embedding of $H^1(\Omega)$ in $L^{p/(p-1)}(\Omega)$ that, for $\psi\in H^1(\Omega)$,
\begin{align*}
	\varepsilon \left| \left\langle \partial_t v_\eta , \psi \right\rangle_{(H^1)',H^1} \right| & = \left| \int_\Omega \nabla v_\eta \cdot \nabla \psi \dd x +\int_\Omega v_\eta \psi \dd x - \int_\Omega u_\eta \psi \dd x \right| \\
	& \le \|v_\eta\|_{H^1} \|\psi\|_{H^1} + \|u_\eta\|_p \|\psi\|_{p/(p-1)} \\ 
	& \le C(p) \left( \|v_\eta\|_{H^1} + \|u_\eta\|_p \right) \|\psi\|_{H^1}\,.
\end{align*}
Estimate~\eqref{es.un05} is then a consequence of \eqref{es.un04}, \eqref{es.un07}, and the above inequality by a duality argument. Finally, estimate~\eqref{es.un03} is obtained thanks to \eqref{es.un04b}, \eqref{es.un06a}, and the upper bound \eqref{bdeta} on $\gamma_\eta$.
\end{proof}
\bigskip

We are now ready to pass to the limit $\eta \rightarrow 0$ and therefore prove the existence of a global weak solution to the system \eqref{ivp}.

\begin{proof}[End of the proof of \Cref{theo26}] Thanks to the uniform estimates collected in \Cref{prop.un0}, we can extract a sequence $(u_{\eta_n},v_{\eta_n})_{n\ge 1}$ such that, for all $T>0$, $q\in (1,\infty)$, and $p\in (2N/(N+2),2)$,
\begin{subequations}\label{cv0}
\begin{align}
u_{\eta_n}  & \rightharpoonup u && \text{in }\; L^p((0,T)\times \Omega), \label{cvu0}\\
\mathcal{K}(u_{\eta_n}-m) & \mathop{\rightharpoonup}^* \mathcal{K}(u-\langle u \rangle) &&\text{in }\; L^\infty((0,T);H^1(\Omega)), \label{cvku0} \\
\partial_t \mathcal{K}(u_{\eta_n}-m)&\rightharpoonup  \partial_t \mathcal{K}(u-\langle u \rangle) &&\text{in }\; L^2((0,T)\times\Omega)), \label{cvdtku0}\\
v_{\eta_n}& \rightharpoonup v &&\text{in }\; L^2((0,T);H^1(\Omega)) \cap L^q((0,T)\times\Omega), \label{cvv0}\\
v_{\eta_n}& \mathop{\rightharpoonup}^* v &&\text{in }\; L^\infty((0,T);L^2(\Omega)), \label{cvv20}\\
\partial_t v_{\eta_n}&\rightharpoonup  \partial_t v &&\text{in }\;  L^p((0,T);(H^1)'(\Omega). \label{cvdtv0}
\end{align}
\end{subequations}
Thanks to \eqref{es.un02}, \eqref{es.un04b}, \eqref{es.un04}, \eqref{es.un06b}, and \eqref{es.un05}, we can furthermore apply Aubin-Lions-Simon Theorem (see \cite[Corollary~4]{simon1986compact}), and extract further subsequences $(\mathcal{K}(u_{\eta_n}-m))_{n\ge 1}$ and $(v_{\eta_n})_{n\ge 1}$ that converges in a strong sense,
\begin{align}
\label{strongcv10}
\mathcal{K}(u_{\eta_n}-m) & \rightarrow \mathcal{K}(u-\langle u \rangle) &&\text{in }\; C([0,T];L^2(\Omega)),\\
v_{\eta_n}  & \rightarrow v && \text{in }\; C([0,T];(H^1)'(\Omega))\cap L^2((0,T)\times\Omega) \;\text{ and a.e. in }\; (0,T)\times\Omega\,. \label{strongcv20}
\end{align}

Next, since $\partial_t u_{\eta_n} = \Delta(u_{\eta_n} \gamma_{\eta_n}(v_{\eta_n}))$ by \eqref{rivp1}, we infer from Lemma~\ref{lem.gameta} and \eqref{es.un06a} that $\big(\partial_t u_{\eta_n}\big)_{n\ge 1}$ is bounded in $L^2((0,T);(H^2)'(\Omega))$, while \eqref{es.un02} guarantees that $(u_{\eta_n})_{n\ge 1}$ is bounded in $L^\infty((0,T);(H^1)'(\Omega))$. Another application of \cite[Corollary~4]{simon1986compact} implies that, after possibly extracting another subsequence, 
\begin{equation}
u_{\eta_n} \rightarrow u \quad \text{in }\; C([0,T];(H^2)'(\Omega)). \label{weakstrongcv0}
\end{equation}
Since $x\mapsto 1$ belongs to $H^2(\Omega)$, an immediate consequence of \eqref{es.un01} and  \eqref{weakstrongcv0} is
\begin{equation}
\langle u(t) , 1 \rangle_{(H^2)',H^2} = \lim_{n\to\infty} \langle u_{\eta_n}(t) , 1 \rangle_{(H^2)',H^2} = |\Omega| \lim_{n\to\infty} \langle u_{\eta_n}(t) \rangle = m |\Omega|\,, \qquad t\in [0,T]\,. \label{pseudo_mass_con0}
\end{equation}
Introducing the space $C([0,T];w-(H^1)'(\Omega))$ of functions from $[0,T]$ to $(H^1)'(\Omega)$ which are continuous with respect to time for the weak topology of $(H^1)'(\Omega)$, we recall that
\begin{equation*}
L^\infty((0,T);(H^1)'(\Omega))\cap C([0,T];(H^2)'(\Omega)) \subset C([0,T];w-(H^1)'(\Omega))\,,
\end{equation*}
and deduce from \eqref{cvku0} and \eqref{pseudo_mass_con0} that
\begin{equation}\label{mass_con0}
\langle u(t) \rangle = m\,,  \qquad t\in [0,T]\,.
\end{equation}
Furthermore, $u\in L^1((0,T)\times \Omega)$ by \eqref{cvu0}. Consequently, $u(t)$ belongs to $L^1(\Omega)$ for a.e. $t\in [0,T]$ which ensures, together with \eqref{mass_con0} and the nonnegativity of $u$, that 
\begin{equation}
u\in L^\infty((0,T);L^1(\Omega))\,. \label{mk7}
\end{equation}
Recalling \eqref{cvu0}--\eqref{cvdtv0} and \eqref{mk7}, we have thus shown that $(u,v)$ satisfies the regularity properties required in Theorem~\ref{theo26}.

We next identify the weak limits of $\big(u_{\eta_n}\sqrt{\gamma_{\eta_n}(v_{\eta_n})}\big)_{n\ge 1}$ and  $\big(u_{\eta_n}\gamma_{\eta_n}(v_{\eta_n})\big)_{n\ge 1}$. To this end, we note that the uniform convergence of $(\gamma_{\eta_n})_{n\ge 1}$ to $\gamma$ on compact subsets of $[0,\infty)$, the bound from Lemma \ref{lem.gameta}, the a.e. convergence \eqref{strongcv20} of $(v_{\eta_n})_{n\ge 1}$, and Lebesgue's convergence theorem imply that, for all $q\in [1,\infty)$,
\begin{equation}
\gamma_{\eta_n}(v_{\eta_n}) \rightarrow \gamma(v) \quad\text{ in }\; L^q((0,T)\times\Omega) \;\text{ and a.e. in }\; (0,T)\times\Omega\,. \label{cgv0}
\end{equation}
It then follows from \eqref{cvu0} and \eqref{cgv0} that 
\begin{align*}
u_{\eta_n} \sqrt{\gamma_{\eta_n}(v_{\eta_n})} &\rightharpoonup u \sqrt{\gamma(v)} &&\text{in }\;  L^1((0,T)\times\Omega)\,, \\
u_{\eta_n} \gamma_{\eta_n}(v_{\eta_n}) &\rightharpoonup u \gamma(v) &&\text{in }\;  L^1((0,T)\times\Omega)\,.
\end{align*}
Since these two sequences are bounded in $L^2((0,T)\times\Omega)$ by \Cref{lem.gameta} and \eqref{es.un06a}, we conclude that
\begin{equation}
\begin{split}
u_{\eta_n} \sqrt{\gamma_{\eta_n}(v_{\eta_n})} \rightharpoonup u \sqrt{\gamma(v)} & \quad\text{in }\;  L^2((0,T)\times\Omega)\,, \\
u_{\eta_n} \gamma_{\eta_n}(v_{\eta_n}) \rightharpoonup u \gamma(v) & \quad\text{in }\;  L^2((0,T)\times\Omega)\,.
\end{split} \label{cvugv0}
\end{equation}

\medskip

 Writing now a very weak formulation of system~\eqref{rivp} like in \Cref{dw2}, we can pass to the limit when $\eta \to 0$ and get that $(u,v)$ indeed are very weak solutions of system~\eqref{ivp} in the sense of \Cref{dw2}.
 
 \medskip
 
Finally, assuming additionally that $v^{in}\in W^{1,q}(\Omega)$ for some $q\in (1,2)$, so that the family $(v_\eta^{in})$ satisfies \eqref{a2b}, we infer from \eqref{strongcv20} and \Cref{lem.a26} that $v\in L^\infty((0,T);W^{1,q}(\Omega))$, which completes the proof.
\end{proof}

\section{Long-term behavior and a Lyapunov functional}\label{sec.ltblf}

This section is devoted to the existence of a Lyapunov functional and the proof of Theorem~\ref{thm.lt}. Let $\Omega$ be a smooth bounded domain of $\RR^N$, with $N\ge 2$ and $\varepsilon>0$. Assume that $\gamma\in C([0,\infty))\cap C^3((0,\infty))$ satisfies \eqref{mon} and consider nonnegative initial conditions $(u^{in},v^{in})\in W^{1,r}(\Omega;\mathbb{R}^2)$ for some $r>N$. It then follows from \cite{JLZ2022, FuSe2021} that there is a unique global classical solution $(u,v)$ to~\eqref{ivp}.  Setting $m := \langle u^{in}\rangle$, it readily follows from~\eqref{ivp} and the nonnegativity of $(u,v)$ that
\begin{subequations}\label{Z0}
\begin{equation}
	\| u(t)\|_1 = |\Omega| \langle u(t) \rangle = m |\Omega|\,, \qquad t\ge 0\,, \label{Z0u}  
\end{equation}
and
\begin{equation}
	\varepsilon \frac{\mathrm{d}}{\mathrm{d}t} \|v(t)\|_1 + \|v\|_1 = m|\Omega|\,, \qquad t\ge 0\,, \label{lf1}
\end{equation}
from which we deduce that
\begin{equation}
	\|v(t)\|_1 = \|u^{in}\|_1 \big( 1-e^{-t/\varepsilon} \big) + \|v^{in}\|_1 e^{-t/\varepsilon} \le \max\big\{ \|u^{in}\|_1 , \|v^{in}\|_1 \big\}\,, \qquad t\ge 0\,.  \label{Z0v}
\end{equation}
\end{subequations}

\subsection{A Lyapunov functional}\label{sec.lf}

\begin{lemma}\label{lem.lf1}
The function $G_0$ defined in~\eqref{G0} is nonnegative and convex on $(0,\infty)$, and 
\begin{equation}
	\frac{\dd}{\dd t} \mathcal{L}_0(u(t),v(t)) + \mathcal{D}_0(u(t),v(t)) = 0\,, \qquad t>0\,, \label{Z1} 
\end{equation}
recalling that $\mathcal{L}_0$ and $\mathcal{D}_0$ are both nonnegative and defined in \eqref{L0} and \eqref{D0}, respectively. In particular, 
\begin{equation}
	\mathcal{L}_0(u(t),v(t)) + \int_0^t \mathcal{D}_0(u(s),v(s)) \dd s \le \mathcal{L}_0(u^{in},v^{in})\,, \qquad t\ge 0\,. \label{Z2}
\end{equation}
\end{lemma}
\begin{proof}[Proof of Lemma~\ref{lem.lf1}]
Since $G_0''(z) = 2 z \gamma'(z) + 2 \gamma(z) - m \gamma'(z) \ge 0$ for $z>0$ by \eqref{mon} and \eqref{G0}, the function $G_0$ is convex on $(0,\infty)$. We then deduce from the convexity of $G_0$ and \eqref{G0} that $G_0(z) \ge G_0(m) + G_0'(m)(z-m)=0$; that is, $G_0$ is nonnegative on $(0,\infty)$. 

\medskip

We next infer from \eqref{ivp1}, \eqref{opk}, and \eqref{lf1} that 
\begin{align}
\frac{1}{2} \frac{\mathrm{d}}{\mathrm{d}t} \|\nabla \mathcal{K}(u-m)\|_2^2 & = - \int_\Omega \mathcal{K}(u-m) \partial_t \Delta \mathcal{K}(u-m)\ \mathrm{d}x = \int_\Omega \mathcal{K}(u-m) \partial_t u\ \mathrm{d}x \nonumber \\
& = \int_\Omega \mathcal{K}(u-m) \Delta (u\gamma(v))\ \mathrm{d}x = \int_\Omega u \gamma(v) \Delta \mathcal{K}(u-m)\ \mathrm{d}x \nonumber \\
& = \int_\Omega (mu-u^2) \gamma(v)\ \mathrm{d}x \nonumber \\
& = \int_\Omega (mu + v^2 - 2 uv) \gamma(v)\ \mathrm{d}x - \int_\Omega (u-v)^2 \gamma(v)\ \mathrm{d}x\,. \label{lf2}
\end{align}
By \eqref{ivp2} and \eqref{ivp3},
\begin{align*}
\int_\Omega (m-2v)u \gamma(v)\ \mathrm{d}x & = \int_\Omega (m-2v) (\varepsilon \partial_t v - \Delta v + v) \gamma(v)\ \mathrm{d}x \\
& = - \varepsilon \int_\Omega (m\gamma(m)+G_0'(v)) \partial_t v\ \mathrm{d}x - \int_\Omega G_0''(v) |\nabla v|^2\ \mathrm{d}x \\
& \qquad + \int_\Omega (mv-2v^2) \gamma(v)\ \mathrm{d}x\,.
\end{align*}
Combining the above identity with \eqref{lf2} gives
\begin{align*}
\frac{1}{2} \frac{\mathrm{d}}{\mathrm{d}t} \|\nabla \mathcal{K}(u-m)\|_2^2 & = - \varepsilon \frac{\mathrm{d}}{\mathrm{d}t} \int_\Omega (m\gamma(m) v + G_0(v))\ \mathrm{d}x - \int_\Omega (u-v)^2 \gamma(v)\ \mathrm{d}x \\
& \qquad - \int_\Omega G_0''(v) |\nabla v|^2\ \mathrm{d}x + \int_\Omega (mv-v^2) \gamma(v)\ \mathrm{d}x\,,
\end{align*}
while we deduce from \eqref{lf1} that
\begin{equation*}
- \varepsilon m \gamma(m) \frac{\mathrm{d}}{\mathrm{d}t} \|v\|_1 - m \gamma(m) \|v\|_1 = - m^2 \gamma(m) |\Omega|\,.
\end{equation*}
Adding the previous two formulas leads us to
\begin{align*}
& \frac{1}{2} \frac{\mathrm{d}}{\mathrm{d}t} \|\nabla \mathcal{K}(u-m)\|_2^2 + \varepsilon \frac{\mathrm{d}}{\mathrm{d}t} \int_\Omega G_0(v)\ \mathrm{d}x \\
& \qquad = - \int_\Omega (u-v)^2 \gamma(v)\ \mathrm{d}x - \int_\Omega G_0''(v) |\nabla v|^2\ \mathrm{d}x - \int_\Omega (v-m) (v\gamma(v)-m\gamma(m))\ \mathrm{d}x\,,
\end{align*}
and we have proved \eqref{Z1}. Now, the nonnegativity of $\mathcal{L}_0$ and $\mathcal{D}_0$ is a consequence of the already established nonnegativity and convexity of $G_0$ and the monotonicity of $z\mapsto z \gamma(z)$ which is due to \eqref{mon}. Finally, the bound~\eqref{Z2} readily follows from~\eqref{Z1} after integration with respect to time.
\end{proof}

We next supplement the bound~\eqref{Z2} with additional estimates on $v$. 

\begin{lemma}\label{lem.lf2}
	\begin{equation}
		\sup_{t\ge 0} \left\{ \|v(t)\|_{H^1}^2 \right\} +  \varepsilon \int_0^\infty \|\partial_t v(s)\|_2^2 \ \mathrm{d}s < \infty\,, \label{Z3}
	\end{equation}
\end{lemma}

\begin{proof}
	We multiply \eqref{ivp2} by $\partial_t v$ and integrate over $\Omega$ to obtain
	\begin{equation*}
		\varepsilon \|\partial_t v\|_2^2 + \frac{1}{2} \frac{\dd}{\dd t} \|v\|_{H^1}^2 = \int_\Omega u \partial_t v\ \mathrm{d}x = \frac{\dd}{\dd t} \int_\Omega u v\ \mathrm{d}x - \int_\Omega v \partial_t u\ \mathrm{d}x \,.
	\end{equation*}
	Then, using again \eqref{ivp},
	\begin{align}
		\varepsilon \|\partial_t v\|_2^2 & + \frac{\mathrm{d}}{\mathrm{d}t} \left( \frac{\|v\|_{H^1}^2}{2} - \|u v\|_1 \right) = - \int_\Omega u \gamma(v) \Delta v\ \mathrm{d}x \nonumber \\
		& = - \int_\Omega (u - v) \gamma(v) \Delta v\ \mathrm{d}x - \int_\Omega v \gamma(v) \Delta v\ \mathrm{d}x\,. \label{a102}
	\end{align}
	On the one hand, by \eqref{ivp3}, \eqref{mon}, and the nonnegativity of $m$, 
	\begin{align}
		- \int_\Omega v \gamma(v) \Delta v\ \mathrm{d}x & = \int_\Omega \left( v \gamma'(v) + \gamma(v) \right) |\nabla v|^2\ \mathrm{d}x = \frac{1}{2} \int_\Omega \big[ G_{0}''(v) + m \gamma'(v) \big] |\nabla v|^2\ \mathrm{d}x  \nonumber \\
		& \le \frac{1}{2} \int_\Omega G_0''(v) |\nabla v|^2\ \mathrm{d}x \,. \label{a103}
	\end{align}
	On the other hand, we infer from \eqref{ivp2}, \eqref{mon}, and Young's inequality that
	\begin{align}
		- \int_\Omega (u - v) \gamma(v) \Delta v\ \mathrm{d}x & = \int_\Omega (u - v) \gamma(v) \left( u - v - \varepsilon \partial_t v \right)\ \mathrm{d}x \nonumber \\
		&\le \int_\Omega (u - v)^2 \gamma(v)\ \mathrm{d}x + \frac{\varepsilon}{2} \|\partial_t v\|_2^2 + \frac{\varepsilon}{2} \gamma(0) \int_\Omega (u - v)^2 \gamma(v)\ \mathrm{d}x \nonumber \\
		& \le \frac{2+\varepsilon \gamma(0)}{2} \int_\Omega (u - v)^2 \gamma(v)\ \mathrm{d}x + \frac{\varepsilon}{2} \|\partial_t v\|_2^2\,. \label{a104}
	\end{align}
	Recalling the definition \eqref{D0} of $\mathcal{D}_0$ which is the sum of three nonnegative terms, it follows from \eqref{a102}, \eqref{a103}, and \eqref{a104} that
	\begin{equation*}
		\frac{\varepsilon}{2} \|\partial_t v\|_2^2  + \frac{\dd}{\dd t} \left( \frac{\|v\|_{H^1}^2}{2} - \|u v\|_1 \right) \le \frac{2+\varepsilon\gamma(0)}{2} \mathcal{D}_0(u,v)\,.
	\end{equation*}
	
	Now, let $t>0$. Integrating the above differential inequality with respect to time over $(0,t)$ and using~\eqref{Z2} give
	\begin{align}
		\varepsilon \int_0^t \|\partial_t v(s)\|_2^2 \ \mathrm{d}s + \|v(t)\|_{H^1}^2 & \le \|v^{in}\|_{H^1}^2 + 2 \|u(t)v(t)\|_1 + (2+\varepsilon\gamma(0)) \int_0^t  \mathcal{D}_0(u(s),v(s))\ \mathrm{d}s \nonumber\\
		& \le \|v^{in}\|_{H^1}^2 + (2+\varepsilon\gamma(0)) \mathcal{L}_0(u^{in},v^{in}) + 2 \|u(t)v(t)\|_1 \,. \label{a106}
	\end{align}
	Owing to \eqref{opk}, \eqref{L0}, \eqref{Z0v}, \eqref{Z2}, and Young's inequality,
	\begin{align}
		\|u(t) v(t)\|_1 & = \int_\Omega (u(t)-m) v(t)\ \mathrm{d}x + m \|v(t)\|_1 = - \int_\Omega v(t) \Delta\mathcal{K}(u(t)-m)\ \mathrm{d}x + m \|v(t)\|_1 \nonumber \\
		& \le \int_\Omega \nabla v(t) \cdot \nabla\mathcal{K}(u(t)-m)\ \mathrm{d}x + m \max\{\| v^{in} \|_1 , \|u^{in}\|_1 \} \nonumber \\
		& \le \frac{1}{4} \|\nabla v(t)\|_2^2 + \|\nabla\mathcal{K}(u(t)-m)\|_2^2 +  m \max\{\| u^{in} \|_1 , \|v^{in}\|_1 \} \nonumber \\
		& \le \frac{1}{4} \| v(t)\|_{H^1}^2 + 2\mathcal{L}_{0}(u(t),v(t)) +  m \max\{\| u^{in} \|_1 , \|v^{in}\|_1 \} \nonumber \\
		& \le \frac{1}{4} \|v(t)\|_{H^1}^2 + 2\mathcal{L}_{0}(u^{in},v^{in}) +  m \max\{\| u^{in} \|_1 , \|v^{in}\|_1 \} \,. \label{a105}
	\end{align}
Combining~\eqref{a106} and~\eqref{a105} leads us to 
\begin{equation*}
	\varepsilon \int_0^t \|\partial_t v(s)\|_2^2 \ \mathrm{d}s + \frac{\|v(t)\|_{H^1}^2}{2} \le \|v^{in}\|_{H^1}^2 + (6+\varepsilon\gamma(0)) \mathcal{L}_0(u^{in},v^{in}) +  2 m \max\{\| u^{in} \|_1 , \|v^{in}\|_1 \} \,,
\end{equation*}
and completes the proof.
\end{proof}

\subsection{Convergence to spatially homogeneous steady states}\label{sec.cshss}

Collecting the outcome of \Cref{lem.lf1} and \Cref{lem.lf2}, we have established the identity~\eqref{lf0} and the estimates~\eqref{lt.00}. We are left with the long-term convergence and begin with some properties of $G_0$ which we gather in the next lemma.

\begin{lemma}\label{lem.x}
There is $K_1>0$ depending only on $\gamma$ such that $\gamma$ satisfies \eqref{phi1} with $k=1$. Moreover,
\begin{equation}
| z\gamma(z) - m\gamma(m)| \le \gamma(0)\, |z-m|\,, \qquad z\in [0,\infty)\,. \label{lt100}
\end{equation}
In addition, recalling that $G_0$ is defined in~\eqref{G0} and is convex on $(0,\infty)$, the function $\sqrt{G_{0}''}\in L^1(0,z)$ for any $z>0$ and its indefinite integral 
\begin{equation*}
g_0(z) := \int_0^z \sqrt{G_{0}''(z_*)}\ \mathrm{d}z_*\,, \qquad z\in [0,\infty)\,,
\end{equation*}
is well-defined and belongs to $C^{0,\frac12}([0,z])$ for all $z>0$.
\end{lemma}

\begin{proof}
Since $\gamma$ satisfies \eqref{mon}, the function $\gamma$ satisfies $z\gamma(z) \ge \gamma(1)>0$ for $z\ge 1$, while the positivity and monotonicity of $\gamma$ on $[0,1]$ implies that $\min_{[0,1]} \gamma=\gamma(1)>0$. Combining these two facts ensures that $\gamma$ satisfies \eqref{phi1} with $k=1$ and $K_1=1/\gamma(1)$. It next follows
from the monotonicity of $\gamma$ that
\begin{equation}
0 \le \frac{\mathrm{d}}{\mathrm{d}z}(z\gamma(z)) = z \gamma'(z) + \gamma(z) \le \gamma(0)\,, \qquad z\ge 0\,. \label{lzgz}
\end{equation}
Integrating the above differential inequality gives \eqref{lt100}.

Next, the convexity of $G_0$ provided by \Cref{lem.lf1} guarantees that $\sqrt{G_0''}$ is well-defined. Using Cauchy-Schwarz inequality, we obtain that, for $z_2>z_1>0$,
\begin{align*}
\int_{z_1}^{z_2} \sqrt{G_{0}''(z)}\ \mathrm{d}z & \le \sqrt{z_2-z_1} \left( \int_{z_1}^{z_2} G_{0}''(z)\ \mathrm{d}z \right)^{1/2} = \sqrt{z_2-z_1} \sqrt{G_{0}'(z_2) - G_{0}'(z_1)} \\
& \le \sqrt{z_2-z_1} \sqrt{2\gamma(0) z_2 + m \gamma(0)}\,.
\end{align*} 
The stated properties of $g_0$ then readily follow from the above inequality, which concludes \Cref{lem.x}.
\end{proof}
\bigskip

\begin{proof}[Proof of Theorem~\ref{thm.lt}]
	As already mentioned, the identity~\eqref{lf0} and the estimates~\eqref{lt.00} are shown in \Cref{lem.lf1} and \Cref{lem.lf2}, respectively, and we now turn to the large time behavior. Since the monotonicity properties of $\gamma$ guarantee that all the terms in $\mathcal{L}_0(u,v)$ and $\mathcal{D}_0(u,v)$ are nonnegative, we infer from \eqref{lt.00} and the Poincar\'e-Wirtinger inequality that
\begin{equation}
\|\mathcal{K}(u(t)-m)\|_{H^1} + \|v(t)\|_{H^1} \le C\,, \qquad t\ge 0\,, \label{lt02}
\end{equation}
\begin{align}
\int_0^\infty \|\partial_t v(s)\|_2^2\ \mathrm{d}s & + \int_0^\infty \int_\Omega \gamma(v) (u-v)^2\ \mathrm{d}x\, \mathrm{d}s < \infty\,, \label{lt03} \\
\int_0^\infty \int_\Omega G_0''(v) |\nabla v|^2\ \mathrm{d}x\, \mathrm{d}s & + \int_0^\infty \int_\Omega (v-m) \big( v\gamma(v) - m \gamma(m) \big) \ \mathrm{d}x\, \mathrm{d}s < \infty\,. \label{lt04} 
\end{align}
We readily infer from \eqref{lt02} and the compactness of the embedding of $H^1(\Omega)$ in $L^2(\Omega)$ that 
\begin{equation}
\{\mathcal{K}(u(t)-m)\,:\, t\ge 0\} \;\text{and }\; \{ v(t)\,:\, t\ge 0\} \;\text{ are compact in }\; L^2(\Omega) \label{lt05a}
\end{equation}
and there are a sequence $(t_j)_{j\ge 1}$ of positive times, $t_j\to\infty$, and $(U_\infty,v_\infty)\in H^1(\Omega,\mathbb{R}^2)$ such that
\begin{equation}
\lim_{j\to\infty} \left( \left\| \mathcal{K}(u(t_j)-m) - U_\infty \right\|_2 + \|v(t_j)-v_\infty\|_2 \right) = 0\,. \label{lt05b}
\end{equation}
Since
\begin{equation*}
\lim_{t\to\infty} \|v(t)\|_1 = m |\Omega|
\end{equation*}
by \eqref{Z0v}, a straightforward consequence of \eqref{lt05b} and the definition of $\mathcal{K}$ is that
\begin{equation}
\langle U_\infty \rangle = 0 \;\;\text{ and } \langle v_\infty \rangle = m |\Omega|\,. \label{lt200}
\end{equation}

For $j\ge 1$ and $s\in [-1,1]$, we set $(u_j,v_j)(s) := (u,v)(s+t_j)$ and first observe that
\begin{align*}
\|v_j(s)-v_\infty\|_2 & \le \|v_j(s) - v_j(0)\|_2 + \|v_j(0) - v_\infty\|_2 \\
& = \|v(s+t_j) - v(t_j)\|_2 + \|v(t_j) - v_\infty\|_2 \\
& \le \left| \int_{t_j}^{s+t_j} \|\partial_t v(\bar{s})\|_2\ \mathrm{d}\bar{s} \right| + \|v(t_j) - v_\infty\|_2 \\
& \le \left( \int_{t_j-1}^{t_j+1} \|\partial_t v(\bar{s})\|_2^2\ \mathrm{d}\bar{s} \right)^{1/2} + \|v(t_j) - v_\infty\|_2 \,.
\end{align*}
Since the right-hand side of the above inequality does not depend on $s\in [-1,1]$ and converges to zero as $j\to\infty$ according to \eqref{lt03} and \eqref{lt05b}, we conclude that
\begin{equation}
\lim_{j\to \infty} \sup_{s\in [-1,1]} \|v_j(s) -v_\infty\|_2 = 0\,. \label{lt06a}
\end{equation}
An immediate consequence of \eqref{lt06a} is that, up to the extraction of a subsequence, we may assume that
\begin{equation}
\lim_{j\to \infty} v_j(s,x) = v_\infty(x) \;\text{ for a.e. }\; (s,x)\in (-1,1)\times \Omega\,. \label{lt06b}
\end{equation}
It next follows from \eqref{lt04} and the monotonicity
 \eqref{mon} 
of $z\mapsto z \gamma(z)$ that
\begin{align*}
& \lim_{j\to\infty} \int_{-1}^1 \int_\Omega (v_j-m) \big( v_j\gamma(v_j) - m \gamma(m) \big)\ \mathrm{d}x\, \mathrm{d}s \\
& \qquad = \lim_{j\to\infty} \int_{t_j-1}^{t_j+1} \int_\Omega (v-m) \big( v\gamma(v) - m \gamma(m) \big)\ \mathrm{d}x\, \mathrm{d}s = 0\,,
\end{align*}
which gives, together with \eqref{lt06b} and Fatou's lemma\,,
\begin{equation}
\int_{-1}^1 \int_\Omega (v_\infty-m) \big( v_\infty \gamma(v_\infty) - m \gamma(m) \big)\ \mathrm{d}x\, \mathrm{d}s = 0\,. \label{lt07a}
\end{equation}
Introducing
\begin{align*}
m_i & := \inf\{ z\in (0,\infty)\,:\, z\gamma(z) = m\gamma(m)\} \,,\\ 
m_s & := \sup\{ z\in (0,\infty)\,:\, z\gamma(z) = m\gamma(m)\} \in [m,\infty]\,, \\
\mathcal{I} & := \{z\in (0,\infty)\,:\, z\gamma(z) = m\gamma(m)\}\,,
\end{align*}
we infer from the boundedness of $\gamma$ and the monotonicity
 \eqref{mon}
 of $z\mapsto z \gamma(z)$ that 
\begin{equation*}
m_i \in (0,m] \;\;\text{ and }\;\; \mathcal{I} = \left\{ 
\begin{array}{lc}
[m_i,m_s]\,, & m_s< \infty\,, \\\
 & \\\
[m_i,\infty)\,, & m_s=\infty\,.
\end{array}\right.
\end{equation*}
Combining this property with \eqref{lt07a} implies in particular that
\begin{equation}
v_\infty(x) \in \mathcal{I} \;\text{ for a.e. }\; x\in\Omega\,. \label{lt07b}
\end{equation}

At this point, either $m_i=m_s=m$ and it readily follows from \eqref{lt07b} that $v_\infty\equiv m$.

Or $m_i\ne m_s$, and the property $z\gamma(z) = m\gamma(m)$ for $z\in \mathcal{I}$ entails that $G_0''(z) = - m \gamma'(z) = m^2 \gamma(m)/z^2$ for $z\in \mathcal{I}$. Introducing 
\begin{equation*}
g(z) := \left\{ 
\begin{array}{cl}
0\,, & z\in [0,m_i) \\
 & \\
\displaystyle{\int_{m_i}^z \sqrt{G_0''(z_*)}\ \mathrm{d}z_* = m \sqrt{\gamma(m)} \ln{(z/m_i)}\,,} & z\in \mathcal{I}\,, \\
 & \\
\displaystyle{\int_{m_i}^{m_s} \sqrt{G_0''(z_*)}\ \mathrm{d}z_* = m \sqrt{\gamma(m)} \ln{(m_s/m_i)}\,,} & z\in (m_s,\infty) \quad (\text{ when }\; m_s<\infty)\,,
\end{array}\right.
\end{equation*}
we infer from \eqref{lt04}, \eqref{lt06a}, and the Lipschitz continuity of $g$ that
\begin{equation}
\lim_{j\to\infty} \int_{-1}^1 \|\nabla g(v_j(s))\|_2^2\ \mathrm{d}s \le \lim_{j\to\infty} \int_{t_j-1}^{t_j+1} \int_\Omega G_0''(v) |\nabla v|_2^2\ \mathrm{d}x\, \mathrm{d}s = 0, \label{lt09}
\end{equation}
and
\begin{equation}
\lim_{j\to \infty} \sup_{s\in [-1,1]} \|g(v_j(s)) - g(v_\infty)\|_2 = 0\,. \label{lt10} 
\end{equation}
Combining \eqref{lt09} and \eqref{lt10} implies that $\nabla g(v_\infty)=0$ a.e. in $\Omega$ and we deduce from \eqref{lt07b}, the connectedness of $\Omega$, and the strict monotonicity of $g$ on $\mathcal{I}$ that there is a unique $\mu\in \mathcal{I}$ such that $v_\infty=\mu$ a.e. in $\Omega$. Recalling that $\langle v_\infty \rangle=m|\Omega|$ by \eqref{lt200}, we conclude that necessarily $\mu=m$. Consequently, $v_\infty\equiv m$ in this case as well, so that, recalling \eqref{lt06a}, we have shown that
\begin{equation}
\lim_{j\to \infty} \sup_{s\in [-1,1]} \|v_j(s) - m\|_2 = 0\,. \label{lt11}
\end{equation}

We next turn to the behaviour of $u$ and the identification of $U_\infty$ in \eqref{lt05b}. On the one hand, for $p\in [1,4N/(3N-2)]\cap [1,2)$, H\"older's inequality gives
\begin{align*}
\int_{-1}^1 \int_\Omega |u_j-v_j|^p\ \mathrm{d}x\, \mathrm{d}s & = \int_{t_j-1}^{t_j+1} \int_\Omega |u-v|^p \gamma(v)^{p/2} \gamma(v)^{-p/2}\ \mathrm{d}x\, \mathrm{d}s \\
& \le \left( \int_{t_j-1}^{t_j+1} \int_\Omega |u-v|^2 \gamma(v) \ \mathrm{d}x\, \mathrm{d}s \right)^{p/2} \left( \int_{t_j-1}^{t_j+1} \int_\Omega \gamma(v)^{-p/(2-p)}\ \mathrm{d}x\, \mathrm{d}s \right)^{(2-p)/2}\,,
\end{align*}
and we infer from \eqref{lt02}, Lemma~\ref{lem.x}, and the continuous embedding of $H^1(\Omega)$ in $L^{p/(2-p)}(\Omega)$ that
\begin{align*}
\int_{t_j-1}^{t_j+1} \int_\Omega \gamma(v)^{-p/(2-p)}\ \mathrm{d}x\, \mathrm{d}s & \le C \int_{t_j-1}^{t_j+1} \int_\Omega (1+v)^{p/(2-p)}\ \mathrm{d}x\, \mathrm{d}s \\
& \le C(p) \left( 1 + \int_{t_j-1}^{t_j+1} \|v(s)\|_{p/(2-p)}^{p/(2-p)}\ \mathrm{d}s \right) \\
& \le C(p) \left( 1 + \sup_{s\ge 0} \|v(s)\|_{H^1}^{p/(2-p)} \right) \\
& \le C(p)\,. 
\end{align*}
Combining the above inequalities leads us to
\begin{equation*}
\int_{-1}^1 \int_\Omega |u_j-v_j|^p\ \mathrm{d}x\, \mathrm{d}s \le C(p) \left( \int_{t_j-1}^{t_j+1} \int_\Omega |u-v|^2 \gamma(v) \ \mathrm{d}x\, \mathrm{d}s \right)^{p/2}\,,
\end{equation*}
which gives, along with \eqref{lt03},
\begin{equation*}
\lim_{j\to\infty} \int_{-1}^1 \int_\Omega |u_j-v_j|^p\ \mathrm{d}x\, \mathrm{d}s = 0\,.
\end{equation*}
Recalling \eqref{lt11}, we end up with
\begin{equation}
\lim_{j\to\infty} \int_{-1}^1 \int_\Omega |u_j-m|^p\ \mathrm{d}x\, \mathrm{d}s = 0\,. \label{lt12}
\end{equation}
On the other hand, it follows from H\"older's inequality that
\begin{align*}
\|\mathcal{K}\partial_t u\|_2 & = \big\| \langle (u-v) \gamma(v) \rangle + \langle v\gamma(v) - m\gamma(m) \rangle + m\gamma(m) - v \gamma(v) + (v-u) \gamma(v) \big\|_2 \\
& \le \sqrt{|\Omega|} \left( \big| \langle (u-v) \gamma(v) \rangle \big| + \big| \langle v\gamma(v) - m\gamma(m) \rangle \big| \right) \\
& \qquad + \big\| m\gamma(m) - v \gamma(v) \big\|_2 + \big\| (v-u) \gamma(v) \big\|_2 \\
& \le \left( 1 + |\Omega| \right) \left( \big\| v\gamma(v) - m\gamma(m) \big\|_2 + \big\| (v-u) \gamma(v) \big\|_2 \right) \,.
\end{align*}
Since (with $K_0 = ||\gamma||_{\infty}$)
\begin{equation*}
\big\| (v-u) \gamma(v) \big\|_2^2 \le K_0 \int_\Omega (v-u)^2 \gamma(v)\ \mathrm{d}x ,
\end{equation*}
and
\begin{align*}
\big\| v\gamma(v) - m \gamma(m)\big\|_2^2 & \le K_0 \int_\Omega \big| v\gamma(v) - m\gamma(m) \big| |v-m|\ \mathrm{d}x \\
& = K_0 \int_\Omega \big( v\gamma(v) - m\gamma(m) \big) (v-m)\ \mathrm{d}x
\end{align*}
by \eqref{mon} and Lemma~\ref{lem.x}, we conclude that
\begin{equation*}
\|\mathcal{K}\partial_t u\|_2^2 \le 2 K_0 (1+|\Omega|)^2 \int_\Omega \left[ (u-v)^2 \gamma(v) + (v-m)\big( v\gamma(v) - m\gamma(m) \big) \right] \ \mathrm{d}x \,.
\end{equation*}
Hence, thanks to \eqref{lt03} and \eqref{lt04}, 
\begin{equation*}
\int_0^\infty \|\mathcal{K}\partial_t u(s)\|_2^2\ \mathrm{d}s < \infty\,,
\end{equation*}
and, since $\mathcal{K}\partial_t u = \partial_t \mathcal{K}(u-m)$, we argue as in the proof of \eqref{lt06a} to deduce from \eqref{lt05b} and the above integrability property that
\begin{equation}
\lim_{j\to\infty} \sup_{s\in [-1,1]} \|\mathcal{K}(u_j(s)-m) - U_\infty\|_2 = 0\,. \label{lt13}
\end{equation}
According to \eqref{lt02} and \eqref{lt13}, we may also assume that, up to the extraction of a subsequence, 
\begin{equation}
\mathcal{K}(u_j-m) \stackrel{*}{\rightharpoonup} U_\infty \;\;\text{ in }\;\; L^\infty([-1,1];H^1(\Omega))\,. \label{lt14}
\end{equation}
We then infer from \eqref{lt12} and \eqref{lt14} that, for any $\varphi\in H^1(\Omega)$, 
\begin{align*}
0 = \lim_{j\to\infty} \int_{-1}^1 \int_\Omega (u_j(s)-m) \varphi\ \mathrm{d}x\, \mathrm{d}s & = \lim_{j\to\infty} \int_{-1}^1 \int_\Omega \nabla\mathcal{K}(u_j(s)-m)\cdot \nabla\varphi\ \mathrm{d}x\, \mathrm{d}s \\
& = \int_{-1}^1 \int_\Omega \nabla U_\infty\cdot \nabla\varphi\ \mathrm{d}x\, \mathrm{d}s = 2 \int_\Omega \nabla U_\infty\cdot \nabla\varphi\ \mathrm{d}x \,,
\end{align*}
which entails, together with \eqref{lt200}, that $U_\infty\equiv 0$. 

We have thus proved that $(0,m)$ is the only cluster point as $t\to\infty$ of $\{\big( \mathcal{K}(u(t)-m),v(t) \big)\:\ t\ge 0\}$ in $L^2(\Omega;\mathbb{R}^2)$. Together with the already established compactness \eqref{lt05a} of this set in $L^2(\Omega;\mathbb{R}^2)$, this property implies that $\big( \mathcal{K}(u(t)-m),v(t) \big)$ converges to $(0,m)$ in $L^2(\Omega;\mathbb{R}^2)$ as $t\to\infty$ and completes the proof of \Cref{thm.lt}.
\end{proof}

\begin{proof} [Proof of Proposition \ref{prop99}]
  According to \cite[Proposition 1.2]{lin_ni_takagi_JDE1988}, there is $d_0>0$ such that if $d\in (0,d_0)$,
  then there exists a nonconstant positive solution $w=w^{(d)} \in C^2(\bom_0)$ of
  \bas
	\left\{ \begin{array}{l}
	0 = d \Delta w - w + w^k
	\qquad \mbox{in } \Omega_0, \\[1mm]
	0 = \nabla w\cdot \mathbf{n}
	\qquad \mbox{on } \pO_0.
	\end{array} \right.
  \eas
  Setting $R_0:=\frac{1}{\sqrt{d_0}}$, and picking $R>R_0$, we then obtain that $d:=\frac{1}{R^2}$ satisfies $d\in (0,d_0)$, and that for
  \bas
	v(x):=w^{(d)}(\sqrt{d} x)
	\quad \mbox{and} \quad
	u(x):=v^k(x),
	\qquad x\in R\Omega_0,
  \eas
  we have $\nabla u\cdot \mathbf{n} = \nabla v\cdot \mathbf{n}=0$ on $\partial (R\Omega_0)$ as well as
  \bas
	\Delta v(x) - v(x) + u(x) 
	= d\Delta w(\sqrt{d}x) - w(\sqrt{d}x) + w^k(\sqrt{d}x) =0
	\qquad \mbox{for all } x\in R\Omega_0
  \eas
  and
  \bas
	\Delta \big(u\gamma(v)\big) = \Delta (v^k \cdot v^{-k})=0
	\qquad \mbox{in } R\Omega_0,
  \eas
  as claimed.
\end{proof}

\section{System with logistic-type growth}\label{sec.ltg}
%
%
%
%
%
%
%
%
\newcommand{\na}{\nabla}
\newcommand{\Del}{\Delta}
\newcommand{\hs}{\hspace*}
In this final part we address the problem (\ref{02}) involving logistic-type zero order degradation.
As our approach in the present section will no longer make use of a comparison argument, we may here employ a somewhat simpler regularization which enforces global solvability at the respective approximate
level by involving a suitably strong damping in the signal production mechanism. More precisely, assuming throughout this
section that $\gamma$, $h$, $u^{in}$ and $v^{in}$ comply with the requirements from Theorem \ref{theo262b},
for $\eta\in (0,1)$ we shall consider
\begin{subequations}\label{02eta}
\begin{align}
	\partial_t \ueta- \Delta (\ueta \gamma(\veta)) = \ueta \, h(\ueta),
	\qquad &(t,x)\in (0,\infty)\times\Omega, \\
	\varepsilon \partial_t \veta=\Delta \veta-\veta + \frac{\ueta}{1+\eta\ueta},
	\qquad &(t,x)\in (0,\infty)\times\Omega, \\
	\nabla (\ueta \gamma(\veta)) \cdot \mathbf{n}  = \nabla \veta \cdot \mathbf{n} =0,
	\qquad & (t,x)\in (0,\infty)\times\partial\Omega, \\
	(\ueta,\veta)(0, \cdot) = (u^{in},v^{in})\,.
	\quad & x\in\Omega,
\end{align}
\end{subequations}
Indeed, by straightforward adaptation from standard arguments from the theory of Keller-Segel type
cross-diffusion systems (see, e.g., \cite{TaWi2017}) it can be seen that each of these problems admits a global classical solution
$(u_\eta,v_\eta)$ with $0\le u_\eta\in C([0,\infty)\times \bom) \cap C^{1,2}((0,\infty)\times \bom)$ and
$0\le v_{\eta} \in \bigcap_{q>1} C([0,\infty);W^{1,q}(\Om)) \cap C^{1,2}((0,\infty)\times\bom)$, and that
\be{v_lower}
	v_\eta(t,x) \ge \left\{ \inf_\Om v^{in} \right\} e^{-t/{\eps}}
	\qquad \mbox{for all $(t,x) \in [0,\infty)\times\bom$.}
\ee
Now the core of this section is contained in the following.
\begin{lemma}\label{lem222}
  Assume (\ref{gamma2}) and (\ref{h}). Then for all $T>0$ there exists $C(T)>0$ such that for all $\eta\in (0,1)$,
  \be{222.1}
	\int_0^T \io \left\{ \ueta \ln (\ueta+e) |h(\ueta)|
	+ \ueta^2
	+ \gamma(\veta) \frac{|\na\ueta|^2}{\ueta+e} 
	+ |\Delta \veta|^2 
	+ \frac{|\na \veta|^4}{\veta^2} \right \} \dd x \dd t
	\le C(T),
  \ee
  as well as
  \be{222.01}
	 \io \ueta(t) \ln \big(\ueta(t)+e\big) \dd x + \io |\na\veta(t)|^2 \dd x \le C(T) \qquad \text{for all } t\in (0,T).
  \ee
\end{lemma}
\proof
  From Assumption (\ref{gamma2}) and (\ref{v_lower}), we know that there exists 
  $c_1(T)>1$ such that
  \be{222.2}
  \veta \ge \frac{1}{c_1(T)}\, , \qquad
  \gamma(\veta) \le c_1(T) \qquad \text{and} \qquad
	\frac{\gamma'^2(\veta)}{\gamma(\veta)} \le \frac{c_1(T)}{\veta}
	\quad \mbox{in } (0,T)\times\Om
	\qquad \mbox{for all } \eta\in (0,1).
  \ee
We furthermore combine \cite[Lemma 3.3]{Wi2012} (with $h(s)=e^{-s}$ and with $e^{-\varphi}$ replaced by $\varphi$) with elliptic regularity theory to find $c_2>0$ fulfilling
  \be{222.3}
	c_2 \io \frac{|\na\varphi|^4}{\varphi^2} \dd x \le \io |\Delta \varphi|^2 \dd x
	\qquad \mbox{for all $\varphi\in C^2(\bom)$ such that $\varphi>0$ in $\bom$ and $\nabla \varphi \cdot \mathbf{n} =0$ on $\partial \Omega$.}
  \ee
  We then define
  \begin{subequations}\label{222.4}
  \be{222.4a}
	\mu(s):=\frac{h_-(s) \ln (s+e)}{s+1},
	\quad s\ge 0,
	\ee
	where
	\be{222.4b}
	h_-(s) := \max\left( -h(s), 0\right), \,\,h_+(s) := h(s) + h_-(s), 
	\quad s\ge 0,
  \ee
   \end{subequations}
  and use (\ref{02eta}) along with Young's inequality and (\ref{222.2}) to see that whenever $\eta\in (0,1)$,
  \bea{222.5}
	& & \hs{-16mm}
	\frac{\dd}{\dd t} \io (\ueta +e) \Big\{ \ln (\ueta+e)-1 \Big\} \dd x
	+ \frac{1}{2} \io \ueta \ln (\ueta+e) |h(\ueta)| \dd x + \io \ueta^2 \dd x \nn\\
	&=& \io \Big\{ \Delta (\ueta\gamma(\veta)) + \ueta h(\ueta) \Big\} \ln (\ueta+e)  \dd x \nn\\
	& & + \frac{1}{2} \io \ueta \ln (\ueta+e) |h(\ueta)| \dd x + \io \ueta^2 \dd x \nn\\
	&=& - \io \gamma(\veta) \frac{|\na\ueta|^2}{\ueta+e} \dd x
	- \io \frac{\ueta}{\ueta+e} \gamma'(\veta) \na\ueta\cdot\na\veta \dd x \nn\\
	& & + \frac{3}{2} \io \ueta\ln (\ueta+e) h_+(\ueta) \dd x
	- \frac{1}{2} \io \ueta(\ueta+1) \mu(\ueta) \dd x
	+ \io \ueta^2 \dd x \nn\\
	&\le& -\frac{1}{2} \io \gamma(\veta) \frac{|\na\ueta|^2}{\ueta+e} \dd x
	+ \frac{1}{2} \io \frac{\ueta^2}{\ueta+e} \frac{\gamma'^2(\veta)}{\gamma(\veta)} |\na\veta|^2 \dd x \nn\\
	& & + \frac{3}{2} \io \ueta\ln (\ueta+e) h_+(\ueta) \dd x
	- \frac{1}{2} \io \ueta(\ueta+1) \mu(\ueta) \dd x
	+ \io \ueta^2 \dd x \nn\\
	&\le& -\frac{1}{2} \io \gamma(\veta) \frac{|\na\ueta|^2}{\ueta+e}\dd x
	+ \frac{c_1(T)}{2} \io \ueta \frac{|\na\veta|^2}{\veta} \dd x \nn\\
	& & + \frac{3}{2} \io \ueta\ln (\ueta+e) h_+(\ueta) \dd x
	- \frac{1}{2} \io \ueta(\ueta+1) \mu(\ueta) \dd x
	+ \io \ueta^2  \dd x \nn\\
	&\le& -\frac{1}{2} \io \gamma(\veta) \frac{|\na\ueta|^2}{\ueta+e} \dd x
	+ \frac{c_2}{2} \io \frac{|\na\veta|^4}{\veta^2} \dd x \nn\\
	& & + \frac{3}{2} \io \ueta\ln (\ueta+e) h_+(\ueta) \dd x
	- \frac{1}{2} \io \ueta(\ueta+1) \mu(\ueta) \dd x \nn\\
	& & + \Big( 1+\frac{c_1^2(T)}{8c_2} \Big) \io \ueta^2 \dd x
	\quad \mbox{for all } t\in (0,T),
  \eea
  as well as
  \bea{222.6}
	\eps \frac{\dd}{\dd t} \io |\na\veta|^2 \dd x
	+ \io |\Delta\veta|^2 \dd x
	+ 2\io |\na\veta|^2 \dd x
	&=& - \io |\Delta\veta|^2 \dd x
	- 2 \io \frac{\ueta}{1+\eta\ueta} \Del\veta \dd x \nn\\
	&\le& \io \ueta^2 \dd x
	\qquad \mbox{for all } t>0.
  \eea
  Since (\ref{222.4}), together with (\ref{h}), ensures that
  \bas
	\frac{3}{2} s \ln (s+e) h_+(s) - \frac{1}{2} s(s+1) \mu(s) + \Big(2+\frac{c_1^2(T)}{8c_2}\Big) s^2 
 \to -\infty
	\qquad \mbox{as } s\to\infty,
  \eas
 and thus there exists $c_3(T)>0$ such that 
  \bas
	\frac{3}{2} s \ln (s+e) h_+(s) - \frac{1}{2} s(s+1) \mu(s) + \Big(2+\frac{c_1^2(T)}{8c_2}\Big) s^2 
 \le c_3(T)
	\qquad \mbox{for all } s\ge 0.
  \eas
  Combining (\ref{222.5}) with (\ref{222.6}) and (\ref{222.3}), we conclude that for
  \bas
	y_\eta(t):= \io (\ueta + e ) \Big( \ln (\ueta+e)-1 \Big)  \dd x +  \eps \io |\na\veta|^2 \dd x,
	\qquad t\in (0,T), \ \eta\in (0,1),
  \eas
  we have
  \bea{222.99}
	& & \hs{-40mm}
	y_\eta'(t)
	+ \frac{1}{2} \io \gamma(\veta) \frac{|\na\ueta|^2}{\ueta+e} \dd x
	+ \frac{1}{4} \io |\Delta\veta|^2 \dd x
	+ \frac{c_2}{4} \io \frac{|\na\veta|^4}{\veta^2} \dd x \nn\\
	& & + \frac{1}{2} \io \ueta \ln (\ueta+e) |h(\ueta)| \dd x + \io \ueta^2  \dd x \nn\\[2mm]
	&\le& c_3(T) |\Om|
	\qquad \mbox{for all } t>0 \mbox{ and } \eta\in (0,1),
  \eea
  which after a time integration shows that
  \bas
	y_\eta(t)\le \io (u^{in}+e) \left\{ \ln (u^{in}+e)-1 \right\}  \dd x +  \eps \|\na v^{in}\|_2^2 +
		c_3(T) T |\Om|
	\qquad \mbox{for all } t\in (0,T) \mbox{ and } \eta\in (0,1),
  \eas
  and thereby implies (\ref{222.01}), while (\ref{222.1}) can be derived by direct integration in (\ref{222.99}).
\qed

\medskip

An immediate consequence of (\ref{222.1}) reveals some integrability feature of expressions related to the fluxes
appearing in the first equation from (\ref{02eta}), here slightly generalized by involving an exponent $\theta$
which can actually be an arbitrary element of $[\frac{1}{2},\infty)$.
\begin{lemma}\label{lem223}
  If (\ref{gamma2}) and (\ref{h}) hold, then for all $\theta\ge\frac{1}{2}$ and each $T>0$,
 there exists $C(\theta,T)>0$ such that
  \be{223.1}
	\int_0^T \io \Big| \nabla \big\{ \ueta \gamma^\theta(\veta) \big\} \Big|^{4/3} \dd x\dd t	\le C(\theta,T),
  \ee
  for all $\eta\in (0,1)$.
\end{lemma}
\begin{proof}
  Let $\eta\in (0,1)$ and $T>0$. Then in
  \begin{align*}
	\int_0^T \io \Big| \na \big\{ \ueta \gamma^\theta(\veta) \big\} \Big|^{4/3} \dd x\dd t
	\le \, & 2^{1/3} \int_0^T \io \gamma^{4\theta/3}(\veta) |\na\ueta|^{4/3} \dd x\dd t\\
	& + 2^{1/3} \theta^{4/3} \int_0^T \io 
		\ueta^{4/3} \gamma^{4(\theta-1)/3}(\veta) |\gamma'(\veta)|^{4/3} |\na\veta|^{4/3} \dd x\dd t,
  \end{align*}
  we twice use Young's inequality to estimate
  \bas
	\int_0^T \io \gamma^{4\theta/3}(\veta) |\na\ueta|^{4/3} \dd x\dd t 
	&=& \int_0^T \io \left( \gamma(\veta) \frac{|\na\ueta|^2}{\ueta+e} \right)^{2/3} (\ueta+e)^{2/3} \gamma^{(4\theta-2)/3} (\veta) \dd x\dd t\\
	&\le& \int_0^T \io \gamma(\veta) \frac{|\na\ueta|^2}{\ueta+e}
	+ \int_0^T \io (\ueta+e)^2 \gamma^{4\theta-2}(\veta) \dd x\dd t,
  \eas
  and
  \bas
	& & \hspace*{-20mm}
	\int_0^T \io 
		\ueta^{4/3} \gamma^{4(\theta-1)/3}(\veta) |\gamma'(\veta)|^{4/3} |\na\veta|^{4/3} \dd x \dd t \\
	&=& \int_0^T \io \left( \frac{|\na\veta|^4}{\veta^2} \right)^{1/3}
		 \ueta^{4/3} \gamma^{4(\theta-1)/3}(\veta) |\gamma'(\veta)|^{4/3} \veta^{2/3} \dd x \dd t \\
	&\le& \int_0^T \io \frac{|\na\veta|^4}{\veta^2} \dd x \dd t 
	+ \int_0^T \io \ueta^2 \frac{\veta \gamma'^2(\veta)}{\gamma(\veta)} \gamma^{2\theta-1}(\veta) \dd x \dd t .
  \eas  
  Collecting the above inequalities and using (\ref{222.1}), (\ref{222.2}) and $\theta \ge 1/2$ gives (\ref{223.1}).
\end{proof}

\medskip

For later reference, let us briefly note some basic information on mass control in the two components.
\begin{lemma}\label{lem2211}
  Assume (\ref{gamma2}) and (\ref{h}). Then there exists $C>0$ such that
  \be{2211.2}
	\io \ueta(t) \dd x \le C
	\qquad
	\mbox{and}
	\qquad
	\io \veta(t) \dd x \le C
	\qquad \mbox{for all $t>0$ and } \eta\in (0,1).
  \ee
\end{lemma}
\begin{proof}
  Since (\ref{h}) particularly entails the existence of $s_1>0$ such that $h(s)\le -1$ for all $s>s_1$, it follows from \eqref{02eta} that
  \bas
	\frac{\dd}{\dd t} \io \ueta \dd x + \io \ueta \dd x 
	&=& \io \ueta \left( 1+ h(\ueta) \right) \dd x 
	\le \io \mathbf{1}_{(0,s_1)}(\ueta) \, \ueta \left( 1+ h(\ueta) \right) \dd x \\
	&\le& s_1 |\Omega| \left( 1 + \sup_{(0,s_1)}|h| \right)
	\qquad \mbox{for all $t>0$ and } \eta\in (0,1),
  \eas
  and that thus, by a simple comparison argument,
  \bas
	\io \ueta (t) \dd x 
	\le c_4:= \max\left\{ \io u^{in} \dd x \, , \, s_1 |\Omega| \left( 1 + \sup_{(0,s_1)}|h| \right) \right\}
	\qquad \mbox{for all $t>0$ and } \eta\in (0,1).
  \eas
From the second equation in (\ref{02eta}) we therefore obtain that
  \bas
	\eps \frac{\dd}{\dd t} \io \veta \dd x = - \io \veta \dd x + \io \frac{\ueta}{1+\eta\ueta} \dd x 
	\le - \io \veta \dd x + c_4
	\qquad \mbox{for all $t>0$ and } \eta\in (0,1),
  \eas
  and a simple time integration completes the proof.
\end{proof}

\medskip

Now the gradient bound from Lemma \ref{lem223} can be supplemented by a time regularity feature:
\begin{lemma}\label{lem224}
  Suppose that (\ref{gamma2}) and (\ref{h}) hold, and 
  let $p> \max(N,4)$. Then for all $T>0$ there exists $C(p,T)>0$ such that
  \be{224.1}
	\int_0^T \Big\| \pa_t \big\{ \ueta(t) \gamma(\veta(t))\big\}\Big\|_{(W^{1,p}(\Om))'} \dd t \le C(p,T),
  \ee
  for all $\eta\in (0,1)$.
\end{lemma}
\begin{proof}
  Since $p>N$ and $p\ge 4$, we can pick $c_5>0$ such that $\|\psi\|_\infty + \|\na\psi\|_4 \le c_5$
  for all $\psi\in C^1(\bom)$ fulfilling $\|\psi\|_{W^{1,p}}\le 1$.
  Fixing any such $\psi$, from (\ref{02eta}) we obtain that for all $t>0$ and $\eta\in (0,1)$,
  \bea{224.2}
	\hs{-4mm}
	\io \pa_t \big\{ \ueta\gamma(\veta)\big\} \, \psi \dd x 
	= - \io \gamma'(\veta) \Big( \na \big\{ \ueta\gamma(\veta)\big\}\cdot\na\veta \Big) \, \psi \dd x 
	- \io \gamma(\veta) \na \big\{\ueta\gamma(\veta)\big\}\cdot\na\psi \dd x 
	+ \io \rho_\eta\psi \dd x ,
  \eea
  where
  \be{224.3}
	\rho_\eta:=\ueta h(\ueta)\gamma(\veta)
	+ \frac1\eps \ueta\gamma'(\veta)\Del\veta
	- \frac1\eps \ueta\veta\gamma'(\veta)
	+ \frac1\eps \frac{\ueta^2}{1+\eta\ueta} \gamma'(\veta).
  \ee
  Now given $T>0$, once more drawing on (\ref{222.2}), we have
 \be{224.5}
	|\gamma'(\veta)|\le c_6(T):=c_1^{3/2}(T)
	\quad \mbox{in } (0,T)\times\Om
	\qquad \mbox{for all } \eta\in (0,1).
  \ee
  According to our choice of $c_5$ and H\"older and Young's inequalities, in (\ref{224.2}) we can therefore estimate
  \bea{224.6}
	\hs{-6mm}
	\bigg| - \io \gamma'(\veta) \Big( \na \big\{ \ueta\gamma(\veta)\big\}\cdot\na\veta \Big) \, \psi \dd x \bigg|
	&\le& c_5 \io |\gamma'(\veta)| \, \Big| \na \big\{ \ueta\gamma(\veta)\big\} \Big| \, |\na\veta| \dd x \nn\\
	&\le& c_5 \io \Big| \na \big\{ \ueta\gamma(\veta)\big\} \Big|^{4/3} \dd x 
	+ c_5 \io \left|\gamma'(\veta)\right|^4 |\na\veta|^4 \dd x \nn\\
	&\le& c_5 \io \Big| \na \big\{ \ueta\gamma(\veta)\big\} \Big|^{4/3} \dd x 
	+ c_5 c_1^4(T) \io \frac{|\na\veta|^4}{\veta^2} \dd x ,
  \eea
  and 
  \bea{224.7}
	\bigg| - \io \gamma(\veta) \na \big\{\ueta\gamma(\veta)\big\}\cdot\na\psi \dd x \bigg|
	&\le& c_5 c_1(T) \bigg\{ \io \Big| \na \big\{ \ueta\gamma(\veta)\big\} \Big|^{4/3} \dd x \bigg\}^{3/4} \nn\\
	&\le& c_5 c_1(T) \io \Big| \na \big\{ \ueta\gamma(\veta)\big\}\Big|^{4/3} \dd x + c_5 c_1(T),
  \eea
  as well as
  \bea{224.8}
	\bigg| \io \rho_\eta \psi \dd x \bigg|
	\le c_5 \io |\rho_\eta| \dd x ,
  \eea
  for all $t\in (0,T)$ and $\eta\in (0,1)$,
  where by (\ref{222.2}), (\ref{224.3}), (\ref{224.5}) and Young's inequality,
  \bas
	\io |\rho_\eta| \dd x 
	&\le& c_1(T) \io \ueta |h(\ueta)| \dd x 
	+\frac{c_6(T)}\eps  \io \ueta |\Del\veta| \dd x 
	+ \frac{c_1(T)}\eps  \io \ueta \sqrt{\veta} \dd x  
	+ \frac{c_6(T)}\eps  \io \ueta^2 \dd x \\
	&\le& c_1(T) \io \ueta\ln (\ueta+e) |h(\ueta)| \dd x 
	+ \io |\Del\veta|^2 \dd x 
	+ \io \veta \dd x 
	+ \Big(\frac{c_6^2(T)}{4 \eps^2} + \frac{c_1^2(T)}{4\eps^2} + \frac{c_6(T)}{\eps}\Big) \io \ueta^2 \dd x 
  \eas
  for all $t\in (0,T)$ and $\eta\in (0,1)$.
  Together with (\ref{224.6})-(\ref{224.8}) inserted into (\ref{224.2}), this shows that with some $c_7(T)>0$,
  for all $t\in (0,T)$ and $\eta\in (0,1)$, we have
  \bas
	\Big\| \pa_t \big\{ \ueta \gamma(\veta)\big\} \Big\|_{(W^{1,p}(\Om))'} 
	&\le& c_7(T) \bigg\{ \io \Big|\na \big\{\ueta\gamma(\veta)\big\} \Big|^{4/3} \dd x 
	+ \io \frac{|\na\veta|^4}{\veta^2} \dd x 
	+ \io |\Del\veta|^2 \dd x 
	+ \io \veta \dd x \nn\\
	& & \hs{20mm}
	+ \io \ueta\ln (\ueta+e) |h(\ueta)| \dd x 
	+ \io \ueta^2 \dd x + 1 \bigg\},
  \eas
  so that (\ref{224.1}) results upon an integration in time using Lemma~\ref{lem222} and Lemma~\ref{lem2211},
  and applying Lemma~\ref{lem223} to $\theta=1$.
\end{proof}

\medskip

The derivation of our main result in this section thereby reduces to an application of an Aubin-Lions lemma:
\begin{proof}[Proof of Theorem \ref{theo262b}]
  From Lemma~\ref{lem223} (with $\theta=1$), Lemma~\ref{lem222} and (\ref{222.2}), it follows that
  \bas
	\Big(\ueta\gamma(\veta)\Big)_{\eta\in (0,1)}
	\mbox{ is bounded in } L^{4/3}((0,T);W^{1,{4/3}}(\Om))
	\qquad \mbox{for all } T>0,
  \eas
  while Lemma \ref{lem224} asserts that if we fix $p> \max(N,4)$, then
  \bas
	\Big(\pa_t \big\{ \ueta\gamma(\veta) \big\} \Big)_{\eta\in (0,1)}
	\mbox{ is bounded in } L^1((0,T);(W^{1,p}(\Om))')
	\qquad \mbox{for all } T>0.
  \eas
  Apart from that, Lemma~\ref{lem222} in conjunction with Lemma~\ref{lem2211} warrants that
  \bas
	(\veta)_{\eta\in (0,1)}
	\mbox{ is bounded in } L^2((0,T);H^2(\Om))
	\qquad \mbox{for all } T>0,
  \eas
  whereas in view of (\ref{02eta}) it is obvious that Lemma~\ref{lem222} (with Lemma~\ref{lem2211}) moreover entails that
  \bas
	(\pa_t \veta)_{\eta\in (0,1)}
	\mbox{ is bounded in } L^2((0,T);L^2(\Om))
	\qquad \mbox{for all } T>0.
  \eas
  Owing to the compactness of the embeddings of $W^{1,4/3}(\Omega)$ and $H^2(\Omega)$, respectively, in $L^{4/3}(\Omega)$ and $H^1(\Omega)$, respectively, and the continuity of the embeddings of $L^{4/3}(\Omega)$ and $H^1(\Omega)$, respectively, in $(W^{1,p}(\Omega))'$ and $L^2(\Omega)$, respectively, two applications of the Aubin-Lions-Simon lemma \cite[Corollary~4]{simon1986compact} thus provide $(\eta_j)_{j\in\N}\subset (0,1)$ such that $\eta_j\searrow 0$ as
  $j\to\infty$, and also provide that for all $T>0$,
  \begin{eqnarray}
	& & u_{\eta_j}\gamma(v_{\eta_j}) \mathop{\to}_{j\to\infty} z
	\qquad \mbox{a.e.~in } (0,\infty)\times\Om \quad \mbox{and in } L^{4/3}((0,T)\times\Omega), 
	\label{262.2} \\
	& & \na\big\{ u_{\eta_j}\gamma(v_{\eta_j})\big\} \mathop{\wto}_{j\to\infty} \na z
	\qquad \mbox{in } L^{4/3}((0,T)\times\Omega;\R^N) \,w
	\qquad \mbox{and}
	\label{262.3} \\
	& & v_{\eta_j} \mathop{\to}_{j\to\infty} v
	\qquad \mbox{in } L^2((0,T);H^1(\Om))
	\mbox{ and a.e.~in } (0,\infty)\times\Om
	\label{262.5} 
  \end{eqnarray}
  for some nonnegative $z\in L^{4/3}((0,T);W^{1,{4/3}}(\Om))$, and some
	$v\in L^2((0,T);H^{2}(\Om)) $
  which satisfies $v>0$ a.e.~in $(0,\infty)\times\Om$, according to (\ref{262.5})
  and (\ref{v_lower}).\\
  Now from (\ref{v_lower}), (\ref{262.2}), (\ref{262.5}) and the positivity of $\gamma$, it is evident that
  \begin{equation}\label{Baudelaire}
  u_{\eta_j}\mathop{\to}_{j\to\infty} u:=\frac{z}{\gamma(v)} \qquad \mbox{a.e.~in }  (0,\infty)\times\Om,
  \end{equation}
so that since 
  \bas
	(\ueta)_{\eta\in (0,1)}
	\mbox{ is bounded in } L^2((0,T)\times\Om)
	\qquad \mbox{for all } T>0,
  \eas
  by Lemma~\ref{lem222},   we infer that 
  \be{262.78}
	u\in L^2((0,T)\times\Om),
  \ee
  and that thanks to the Vitali convergence theorem,  
  \be{262.6}
	u_{\eta_j} \mathop{\to}_{j\to\infty}  u
	\quad \mbox{in } L^1((0,T)\times\Omega)
	\qquad \mbox{for all } T>0.
  \ee
Similarly, the $L^1$ estimate for $\big(\ueta\ln (\ueta+e)h(\ueta)\big)_{\eta\in (0,1)}$
  contained in (\ref{222.1}) can readily be seen to entail that
  \bas
	\big(\ueta h(\ueta)\big)_{\eta\in (0,1)}
	\mbox{ is uniformly integrable over $(0,T)\times\Om$ \qquad for all $T>0$,}
  \eas
while the continuity of $h$ and \eqref{Baudelaire} guarantees that
  \begin{equation*}
	u_{\eta_j} h(u_{\eta_j}) \mathop{\to}_{j\to\infty} u h(u) \qquad \mbox{a.e.~in }  (0,\infty)\times\Om.
  \end{equation*}
Another application of the Vitali convergence theorem then gives
    \be{262.7}
	u_{\eta_j} h(u_{\eta_j}) \mathop{\to}_{j\to\infty} u h(u)
	\quad \mbox{in } L^1((0,T)\times\Omega).
  \ee
  
  With the regularity requirements in Definition~\ref{dw3} hence being asserted in view of (\ref{262.2}), (\ref{262.3}), (\ref{262.78}) and (\ref{262.7}),
  the derivation of the identities in (\ref{wu3}) and (\ref{wv3}) can be achieved by taking $\eta=\eta_j\searrow 0$ in the corresponding weak formulation of (\ref{02eta}) and using the convergence properties in (\ref{262.3}), (\ref{262.5}), (\ref{262.6}) and (\ref{262.7}).
  To finally verify the claimed additional regularity features, we observe that (\ref{262.1}) follows from $v \in L^2(H^2)$ and (\ref{262.78}) upon observing that the inclusions
  $u\in L^\infty((0,T);L\log L(\Om))$ and $v\in L^\infty((0,T); H^1(\Om))$ for all $T>0$
  are immediate consequences of (\ref{222.01}).	
  The coupled weak differentiability property in (\ref{262.11}) can be concluded from Lemma \ref{lem223} when applied to $\theta=\frac{1}{2}$
  and combined with the fact that $u_{\eta_j}\sqrt{\gamma(v_{\eta_j})} \to u\sqrt{\gamma(v)}$ a.e.~in $(0,\infty)\times \Om$ as
  $j\to \infty$, the latter resulting from (\ref{262.5}), (\ref{Baudelaire}) and the continuity of $\gamma$.
\end{proof}

\section*{Acknowledgments}
The fourth author acknowledges support of the {\em Deutsche For\-schungs\-gemein\-schaft} in the context of the project 
  {\em Emergence of structures and advantages in cross-diffusion systems} (No.~411007140, GZ: WI 3707/5-1).
\bibliographystyle{siam}
\bibliography{CLSWS}

\end{document}